\documentclass[11pt,a4paper,twoside]{article}

\usepackage{amsmath,amssymb,amsthm}
\usepackage{mathrsfs}
\usepackage{xcolor}
\usepackage[hidelinks]{hyperref}
\numberwithin{equation}{section}

\usepackage{titlesec}
\titleformat{\subsection}[runin]
  {\normalfont\bfseries}
  {\thesubsection}{0.75em}{}[.]
\titlespacing*{\subsection}{0pt}{1.5ex plus .5ex minus .2ex}{0.75em}

\usepackage[
left=1.08in,
right=1.08in,
top=1.10in,
bottom=1.10in
]{geometry}

\newcommand{\ddc}{\sqrt{-1}\partial\bar\partial}

\newtheorem{theorem}{Theorem}[section]
\newtheorem{proposition}[theorem]{Proposition}
\newtheorem{lemma}[theorem]{Lemma}

\theoremstyle{definition}

\theoremstyle{remark}
\newtheorem{remark}[theorem]{Remark}

\newcommand{\ii}{\sqrt{-1}}
\newcommand{\ee}{\varepsilon}

\newcommand{\tr}{\textup{tr}}

\usepackage{fancyhdr}

\title{Explicit Ricci-flat Metrics on Kummer K3 Surfaces}

\author{Jixiang Fu, Jiuhe Liu, Yunyang Xiao and Shing-Tung Yau}
\date{}

\begin{document}

\maketitle 

\begin{abstract}
We give an explicit convergent representation of the Ricci-flat
K\"ahler metrics supplied by the Calabi--Yau theorem in a
one-parameter family of classes on a fixed Kummer K3 surface.
An explicit radial background has normalized Monge--Amp\`ere
residual \(O(a^4)\). Its scalar Green operator \(G_a\) is
represented by periodic Ewald kernels, separated radial kernels,
finite-dimensional Schur complements, and convergent Neumann series.
The coefficients are defined by \(U_{a,1}=-G_af_a\) and
\(U_{a,n}=G_a\sum_{j=1}^{n-1}Q_a(U_{a,j},U_{a,n-j})\), where
\(Q_a\) is the polarized quadratic Monge--Amp\`ere term.
A Green estimate of order \(a^{-1}\) bounds the first correction
by \(O(a^3)\) and the convergence parameter by \(O(a^2)\).
A bound using Catalan numbers proves absolute convergence on the entire smooth
surface for every sufficiently small fixed \(a\). The sum defines a positive form satisfying the volume equation.
Uniqueness identifies it with the smooth normalized Calabi--Yau potential.
The first remainder admits scale-adapted bounds in every
fixed derivative order, and the leading \(a^4\) coefficient on
regular compact sets is a cutoff-independent periodic Green sum.
{We have also completed the corresponding construction on
collapsing elliptic K3 surfaces using the Gross--Wilson background.
We state this extension and outline its analytic mechanism; the detailed
elliptic Green construction and estimates are omitted here.}
\end{abstract}

\section{Introduction}
\label{sec:introduction}
Yau discussed the problem of constructing an explicit metric
with holonomy group \(SU(2)\) on a K3
surface~\cite[pp.~333--334]{YauGeometricStructure}.
His theorem~\cite{Yau1978} guarantees a unique Ricci-flat K\"ahler
metric in each K\"ahler class. For a family of classes near the orbifold limit on a fixed Kummer
K3 surface, we give a convergent series for the Calabi--Yau potential.
We construct the background Green operator and compose it with
the quadratic Monge--Amp\`ere term to determine every correction
recursively. The main theorem proves that the sum yields the metric
supplied by the Calabi--Yau theorem in the prescribed K\"ahler class.

Let \(T=\mathbb C^2/(\mathbb Z+\ii\mathbb Z)^2\), and let
\(\pi:X\to T/\{\pm1\}\) be the minimal resolution. Denote the
exceptional curves by \(E_1,\ldots,E_{16}\), and the pullback of
the flat orbifold K\"ahler class by \(\kappa\). We use
\(d^c=\frac{\ii}{2}(\bar\partial-\partial)\), so that
\(dd^c=\ii\partial\bar\partial\). The form \(dz_1\wedge dz_2\)
extends to a nowhere-vanishing holomorphic two-form \(\Omega\)
on \(X\). Set \(\mu=\frac14\Omega\wedge\bar\Omega\), with
\(\int_X\mu=\kappa^2/2=1/2\). For the Eguchi--Hanson scale
\(a>0\), define
\begin{equation}\label{eq:intro-kummer-class}
\alpha_a=\kappa-\frac{\pi a^2}{2}\sum_{\nu=1}^{16}[E_\nu],
\qquad c_a=1-8\pi^2a^4.
\end{equation}
Here \([E_\nu]\) is the Poincar\'e dual class; hence
\(\alpha_a\cdot E_\nu=\pi a^2\) and
\(\alpha_a^2/2=c_a\int_X\mu\).

The classical Kummer construction joins the flat quotient metric
to Eguchi--Hanson metrics~\cite{EguchiHanson} near the sixteen
exceptional curves.
An early approximate metric on a K3 surface appears in Gibbons
and Pope's appendix
\emph{An Approximate Metric on K3}~\cite[Appendix]{GibbonsPope1979}.
The analytic realization of the Kummer construction is treated by
Kobayashi~\cite{Kobayashi1990} and, as a model gluing problem, by
Donaldson~\cite{DonaldsonKummer}; see also Joyce~\cite{Joyce2000}.
We choose an explicit radial interpolation giving a smooth K\"ahler
form \(\omega_a^{\mathrm{mod}}\in\alpha_a\) for small \(a\).
It is exactly Eguchi--Hanson near the exceptional curves and exactly
flat outside fixed neighborhoods. Its volume form
\(dV_a=(\omega_a^{\mathrm{mod}})^2/2=(1-\beta_a)\mu\) has
an explicit density, with \(\beta_a\) supported in the transition
annuli. The normalized residual is
\[
f_a:=\frac{c_a\mu}{dV_a}-1
=\frac{\beta_a-8\pi^2a^4}{1-\beta_a}.
\]
It has zero \(dV_a\)-mean and satisfies
\(\|f_a\|_{C_{\mathrm{sc}}^{0,\alpha}}\le Ca^4\), where
\(0<\alpha<1\) is fixed and the scale-adapted H\"older norms
are defined in Section~\ref{subsec:scale-adapted-norms}.

The background shows that \(\alpha_a\) is K\"ahler for the
parameter range used in its construction. Yau's theorem then supplies a smooth potential
\(v_a\), normalized by \(\int_Xv_a\,dV_a=0\), such that
\begin{equation}\label{eq:intro-Yau-target}
\omega_{\mathrm{CY},a}=\omega_a^{\mathrm{mod}}+\ddc v_a>0,
\qquad \frac{\omega_{\mathrm{CY},a}^2}{2}=c_a\mu.
\end{equation}
No estimate for \(v_a\) is assumed. The bounds for its recursive
representation follow from the residual and the background Green operator.
Yau's theorem applies whenever \(\alpha_a\) is K\"ahler; the
smaller range established below concerns convergence of the
particular expansion about \(\omega_a^{\mathrm{mod}}\).

All potentials and sources are real-valued. On functions of zero
\(dV_a\)-mean, let \(A_a=-\Lambda_{\omega_a^{\mathrm{mod}}}\ddc\)
and \(G_a=A_a^{-1}\), normalized by \(\int_XG_af\,dV_a=0\).
The symmetric bilinear map
\(Q_a(u,v)=\ddc u\wedge\ddc v/(\omega_a^{\mathrm{mod}})^2\)
takes values in the zero-mean source space by Stokes' theorem. In complex dimension two the
volume equation is exactly quadratic:
\[
-A_au+Q_a(u,u)=f_a,
\qquad u=-G_af_a+G_aQ_a(u,u).
\]
Thus the same explicitly represented operator \(G_aQ_a\)
determines every order of correction.

\begin{theorem}[Explicit representation of the Calabi--Yau metric]
\label{thm:main}
There exists \(a_0>0\) such that, for every \(0<a<a_0\), the
mean-zero functions defined by
\begin{equation}\label{eq:intro-recursion}
\begin{aligned}
U_{a,1}&=-G_af_a,\\
U_{a,n}&=G_a\!\left(\sum_{j=1}^{n-1}
Q_a(U_{a,j},U_{a,n-j})\right),\qquad n\ge2,
\end{aligned}
\end{equation}
form an absolutely convergent series in
\(\mathcal X_a=C^{2,\alpha}(X,\mathbb R)/\mathbb R\), endowed
with the norm
\(\|u\|_{\mathcal X_a}=\|(\omega_a^{\mathrm{mod}})^{-1}
\ddc u\|_{C_{\mathrm{sc}}^{0,\alpha}}\).
Its sum is the normalized Calabi--Yau potential \(v_a\), and
\begin{equation}\label{eq:main-metric}
v_a=\sum_{n=1}^{\infty}U_{a,n},
\qquad \omega_{\mathrm{CY},a}
=\omega_a^{\mathrm{mod}}+\ddc\sum_{n=1}^{\infty}U_{a,n}.
\end{equation}
In particular, \(\|v_a\|_{\mathcal X_a}\le Ca^3\). For constants
\(C_1,C_2>0\), independent of \(a\) and \(n\),
\[
\|U_{a,n}\|_{\mathcal X_a}
\le\operatorname{Cat}_{n-1}C_1^nC_2^{n-1}a^{2n+1},
\qquad \operatorname{Cat}_m=\frac1{m+1}\binom{2m}{m}.
\]
The background, residual, and Green operator are specified in
Sections~\ref{sec:kummer-geometry}--\ref{sec:green-construction}
and Appendices~\ref{app:ewald}--\ref{app:EH-kernels}.
\end{theorem}

The convergence is global on the smooth surface \(X\), including
the exceptional curves, for each sufficiently small fixed \(a\).
For fixed \(a\), the displayed norm is equivalent to the ordinary
\(C^{2,\alpha}\)-norm modulo constants. The terms depend on
\(a\) through \(f_a,G_a,Q_a\); the theorem does not assert a
power series in \(a\) with parameter-independent coefficients.

The linear construction cuts along sixteen fixed interfaces,
separating a flat exterior from sixteen caps. Let \(\mathcal D_a\)
be the block Dirichlet inverse, \(\mathcal H_a\) the harmonic
extension, and \(\mathcal J_a\) the jump of outward normal
derivatives. With \(\mathcal N_a=\mathcal J_a\mathcal H_a\),
\[
G_af=(I-\Pi_a)\left(\mathcal D_af
-\mathcal H_a\mathcal N_a^\dagger\mathcal J_a\mathcal D_af\right),
\]
where \(\Pi_av=\int_Xv\,dV_a/\int_XdV_a\). The interface
inverse acts on compatible data modulo the common constant.
Periodic Fourier--Ewald kernels describe the exterior, and angular
separation gives cap kernels through convergent Frobenius and
Volterra series. High-mode Neumann series and finite-dimensional
Schur complements determine the remaining inverses. Each
application of \(G_a\) is therefore prescribed by volume and
boundary integrations, convergent series, and finite-dimensional
{linear algebra. Only the radial kernel construction is restricted
to one-dimensional integrations. Certified numerical evaluation
requires the additional bounds discussed in
Section~\ref{sec:computational-realization}.}

The estimates governing the nonlinear series are
\[
\|G_af\|_{\mathcal X_a}\le Ca^{-1}\|f\|_{\mathcal Y_a},
\qquad \|Q_a(u,v)\|_{\mathcal Y_a}
\le\tfrac12\|u\|_{\mathcal X_a}\|v\|_{\mathcal X_a},
\]
where \(\mathcal Y_a=C_{\mathrm{sc}}^{0,\alpha}(X)\) and
\(f\) has zero mean. They give
\(\|G_af_a\|_{\mathcal X_a}=O(a^3)\) and
\(\|G_aQ_a\|_{\mathrm{bil}}=O(a^{-1})\), so the sufficient
convergence condition
\(4\|G_af_a\|_{\mathcal X_a}\|G_aQ_a\|_{\mathrm{bil}}<1\)
holds for small \(a\). Proposition~\ref{prop:abstract-recursion}
{then gives a bound using Catalan numbers.} Positivity follows along the
convergent path \(\sum_{n\ge1}t^nU_{a,n}\), \(0\le t\le1\),
{under the same convergence condition.} Integration by parts
identifies the sum with \(v_a\), and smoothness follows from
{that identification.}

Explicit K3 expansions near torus orbifold limits appear in
Kachru--Tripathy--Zimet~\cite[Section~3.4]{KTZ2020} and
Tripathy--Zimet~\cite[Section~3.2]{TZ2020}.
Equations (3.77), (3.80), and (3.84) of the latter
give a coefficient recursion in which a fixed linear solution
operator is applied to quadratic combinations of lower-order
corrections; its integral
equation (3.93) has specified kernels. The discussion following
that equation proposes a contraction argument for rigorous
justification. These recursions concern hyper-K\"ahler quotient
data on the dual torus.
Our recursion acts on scalar K\"ahler potentials on the resolved
compact Kummer surface. Donaldson's scalar perturbation
argument~\cite[Sections~3.1--3.2]{DonaldsonKummer} uses a bounded
linear inverse and the inverse function theorem to solve the
quadratic Monge--Amp\`ere equation. His operator \(Q\) is linear
and takes values in real \((1,1)\)-forms, whereas \(Q_a\) above
is a scalar-valued symmetric bilinear map. His argument does not
give the coefficient recursion \eqref{eq:intro-recursion} or the
explicit global Green representation used here.

{Hanson and Sha~\cite{HansonSha2009} construct symmetric
tessellations of Fermat surfaces in \(\mathbb{CP}^3\), including
a twenty-four-vertex tessellation of the quartic Fermat K3 surface.
Their work provides a complementary combinatorial description
of K3 geometry.}

{The contribution here is a global kernel representation of the
bilinear operator \(\mathcal B_a=G_a\circ Q_a\) on the resolved
compact surface, together with a proof that the resulting
scalar-potential expansion converges {to the normalized Calabi--Yau potential}. Its binary-tree form also
gives the truncation bound \eqref{eq:tree-tail}.
}{
Theorem~\ref{thm:global-first-correction} further gives a global
first-term metric approximation with remainder \(O(a^5)\) in
the stated relative H\"older norm and volume residual \(O(a^6)\).
Proposition~\ref{prop:higher-first-remainder} controls the metric
remainder in every fixed scale-adapted derivative order and improves
it to \(O(a^6)\) on fixed regular compact sets.
On the same compact sets,
Theorem~\ref{thm:regular-first-coefficient} identifies the leading
\(a^4\) coefficient with the sixteen-point periodic Green sum;
the expansion holds in every fixed \(C^{k,\alpha}\)-norm,
and the coefficient is independent of the radial cutoff.}

The representation also provides geometric input for hyper-K\"ahler
rotation and special Lagrangian geometry on \(X\). With
\(\widehat\Omega_a=\sqrt{c_a}\,\Omega\), the forms
\((\omega_{\mathrm{CY},a},\operatorname{Re}\widehat\Omega_a,
\operatorname{Im}\widehat\Omega_a)\) determine the hyper-K\"ahler
triple, and its complex structures are obtained pointwise from the
explicitly represented metric. {For either elliptic fibration induced by projection onto a factor
of \(T\), every smooth fiber becomes a special Lagrangian torus
after rotating to a complex structure on the equator perpendicular
to the original complex structure. For each such rotation, the
phase is the same for all smooth fibers.} This
special Lagrangian assertion follows from standard hyper-K\"ahler
{rotation. In the original resolution charts, the metric
representation can be used to formulate holomorphic-disc equations
for further study.}
Disc existence and counts, quantitative coordinate expansions,
and singular completion of the SYZ mirror require further
geometric and analytic arguments.

{The Green-operator method is also developed for two
Calabi--Yau threefold models in the companion manuscripts
\cite{KummerThreefoldMetric,SchoenThreefoldMetric}.
The Kummer threefold construction uses Calabi's local metrics at
twenty-seven isolated order-three quotient singularities; the special
Schoen construction uses Eguchi--Hanson product metrics along sixteen
disjoint elliptic singular curves. Both papers give explicit convergent
representations of the Ricci-flat metrics in their prescribed
one-parameter class families. In complex dimension three, the exact
Monge--Amp\`ere equation contains quadratic and cubic terms, so the
recursion includes both polarizations. These constructions provide
metric input for studying special Lagrangian geometry; a global special
Lagrangian torus fibration is not asserted here.}

Section~\ref{sec:kummer-geometry} constructs the background and
residual. Sections~\ref{sec:green-construction} and
\ref{sec:green-estimates} give the Green representation and its
estimates. Section~\ref{sec:nonlinear-correction} sums the recursion
and identifies its limit with the Calabi--Yau potential.
{It also gives the first-term estimates stated above.}
Section~\ref{sec:computational-realization} discusses evaluation and
truncation. {Finally,
Section~\ref{sec:elliptic-outlook} states the completed extension to
collapsing elliptic K3 surfaces and outlines the same recursive mechanism.
The detailed elliptic Green construction and the proofs of the
corresponding global estimates are omitted from this paper.} The appendices provide the
periodic and radial kernel formulas.

{The elliptic discussion starts from the Gross--Wilson
background~\cite{GrossWilson}; its semi-flat and Ooguri--Vafa
models are described in Section~\ref{sec:elliptic-outlook}.}

{\bfseries Acknowledgments.} 
{The authors used AI to help explore proof strategies and calculations.} 
The authors are solely responsible for verifying all arguments and calculations and for the final content.

\section{Kummer geometry and the explicit background}
\label{sec:kummer-geometry}
We fix the geometric normalizations and construct a radial K\"ahler
background which is exactly Eguchi--Hanson near the sixteen
exceptional curves and exactly flat outside fixed neighborhoods.
Its normalized Calabi--Yau residual is \(O(a^4)\).

\subsection{The flat Kummer orbifold and its resolution}
\label{subsec:kummer-orbifold}
Let \(T=\mathbb C^2/(\mathbb Z+\ii\mathbb Z)^2\) be the square
complex two-torus. In the coordinates \(z_j=x_j+\ii y_j\),
\(j=1,2\), on its universal cover, set
\begin{equation}\label{eq:flat-kahler-form}
\omega_0=\frac{\ii}{2}
\left(dz_1\wedge d\bar z_1+dz_2\wedge d\bar z_2\right).
\end{equation}
The holomorphic form \(\Omega_0=dz_1\wedge dz_2\) satisfies
\begin{equation}\label{eq:flat-volume-identity}
dV_0:=\frac{\omega_0^2}{2}=\frac14\Omega_0\wedge\bar\Omega_0,
\qquad \int_TdV_0=1.
\end{equation}
The involution \(\iota(z)=-z\) has sixteen fixed points,
\(T[2]=(\frac12\mathbb Z/\mathbb Z)^4\). Their images
\(p_1,\ldots,p_{16}\) in \(Y=T/\{\pm1\}\) are
\(A_1\)-singularities, locally isomorphic to
\(\mathbb C^2/\{\pm1\}\). Let \(\pi:X\to Y\) be the minimal
resolution, with exceptional curves \(E_1,\ldots,E_{16}\).
These are mutually disjoint smooth rational curves. The pullback
\(\kappa\in H^{1,1}(X,\mathbb R)\) of the flat orbifold class
satisfies
\begin{equation}\label{eq:intersection-normalizations}
E_\nu^2=-2,\quad E_\nu\cdot E_\mu=0\ (\nu\ne\mu),
\quad \kappa\cdot E_\nu=0,\quad \frac{\kappa^2}{2}=\frac12.
\end{equation}
The invariant form \(\Omega_0\) descends to the regular part
of \(Y\) and extends across its crepant resolution to a
nowhere-vanishing holomorphic two-form \(\Omega\) on \(X\).
Thus \(\mu=\frac14\Omega\wedge\bar\Omega\) has total volume
\(\int_X\mu=1/2\), since \(T\to Y\) has degree two.

\subsection{The Kummer K\"ahler classes}
\label{subsec:kummer-classes}
For the Eguchi--Hanson scale \(a>0\), define
\begin{equation}\label{eq:alpha-a}
\alpha_a=\kappa-\frac{\pi a^2}{2}
\sum_{\nu=1}^{16}[E_\nu].
\end{equation}
Here \([E_\nu]\) is the Poincar\'e dual class.
The intersection identities give \(\alpha_a\cdot E_\nu=\pi a^2\)
and
\begin{equation}\label{eq:alpha-volume}
\frac{\alpha_a^2}{2}
=\frac{\kappa^2}{2}
+\frac{\pi^2a^4}{8}\sum_{\nu=1}^{16}E_\nu^2
=\frac12-4\pi^2a^4.
\end{equation}
Consequently, the target volume constant is
\(c_a=\alpha_a^2/(2\int_X\mu)=1-8\pi^2a^4\), and the
Calabi--Yau equation is \([\omega]=\alpha_a\),
\(\omega^2/2=c_a\mu\). The construction below gives a
K\"ahler representative on an explicit parameter range; it does
not determine the full intersection of this family with the
K\"ahler cone.

\subsection{The Eguchi--Hanson model}
\label{subsec:EH-model}
On the punctured cover of \(\mathbb C^2/\{\pm1\}\), put
\(\rho=|z_1|^2+|z_2|^2\). In the normalization of
\eqref{eq:flat-kahler-form}, the Eguchi--Hanson potential~\cite{EguchiHanson} is
\begin{equation}\label{eq:EH-potential}
\Phi_{\mathrm{EH},a}(\rho)
=\frac12\left[\sqrt{\rho^2+a^4}
+a^2\log\frac{\rho}{\sqrt{\rho^2+a^4}+a^2}\right].
\end{equation}
Set \(\omega_{\mathrm{EH},a}=\ii\partial\bar\partial
\Phi_{\mathrm{EH},a}\). {The potential's} complex Hessian has tangential
and radial eigenvalues
\[
A_{\mathrm{EH},a}=\Phi_{\mathrm{EH},a}'
=\frac{\sqrt{\rho^2+a^4}}{2\rho},
\qquad B_{\mathrm{EH},a}=\Phi_{\mathrm{EH},a}'
+\rho\Phi_{\mathrm{EH},a}''
=\frac{\rho}{2\sqrt{\rho^2+a^4}}.
\]
Primes denote differentiation in \(\rho\). Since
\(4A_{\mathrm{EH},a}B_{\mathrm{EH},a}=1\),
\begin{equation}\label{eq:EH-volume}
\frac{\omega_{\mathrm{EH},a}^2}{2}=\mu.
\end{equation}
For fixed \(a>0\), its expansion at \(\rho=0\) is
\[
\Phi_{\mathrm{EH},a}(\rho)
=\frac{a^2}{2}\log\rho
+\frac{a^2}{2}\bigl(1-\log(2a^2)\bigr)
+\frac{\rho^2}{8a^2}+O(\rho^4/a^6).
\]
In the resolution chart \(s=z_1^2\), \(\xi=z_2/z_1\), one has
\(\rho=|s|(1+|\xi|^2)\) and \(E=\{s=0\}\).
Subtracting the pluriharmonic term \(\frac{a^2}{2}\log|s|\)
gives a potential extending smoothly across \(E\). Its restriction {to the exceptional curve}, up to an additive constant,
is \(\frac{a^2}{2}\log(1+|\xi|^2)\); the Fubini--Study
form \(\ii\partial\bar\partial\log(1+|\xi|^2)\) has integral
\(2\pi\), so
\begin{equation}\label{eq:EH-exceptional-area}
\int_E\omega_{\mathrm{EH},a}=\pi a^2.
\end{equation}
This fixes the coefficient in \eqref{eq:alpha-a}.

For \(\rho>a^2\), the potential has the convergent expansion
\begin{equation}\label{eq:EH-asymptotic-series}
\Phi_{\mathrm{EH},a}(\rho)
=\frac{\rho}{2}+\sum_{j=1}^{\infty}c_ja^{4j}\rho^{1-2j},
\qquad c_j=(-1)^j\frac{(2j-3)!!}{2^{j+1}j!(2j-1)}.
\end{equation}
Here \(j\ge1\) and \((-1)!!=1\). The first terms are
\[
\Phi_{\mathrm{EH},a}(\rho)
=\frac{\rho}{2}-\frac{a^4}{4\rho}
+\frac{a^8}{48\rho^3}-\frac{a^{12}}{160\rho^5}
+\frac{5a^{16}}{1792\rho^7}+O(a^{20}\rho^{-9}).
\]
The coefficient expansion is useful for evaluation, but the
construction uses the exact potential \eqref{eq:EH-potential}.

\subsection{The radial interpolation}
\label{subsec:radial-interpolation}
Fix radii
\begin{equation}\label{eq:fixed-radii}
0<r_1<r_2<r_3<r_4<R<\frac18.
\end{equation}
The sixteen orbifold balls of radius \(R\) are pairwise disjoint.
For an explicit cutoff, set \(\vartheta(t)=e^{-1/t}\) for
\(t>0\) and \(\vartheta(t)=0\) for \(t\le0\), and let
\(\chi(t)=\vartheta(t)/(\vartheta(t)+\vartheta(1-t))\).
Then \(\chi\) is smooth, vanishes on \((-\infty,0]\), and
equals one on \([1,\infty)\). Define
\[
\psi(r)=\chi\!\left(\frac{r-r_1}{r_2-r_1}\right)
\chi\!\left(\frac{r_4-r}{r_4-r_3}\right).
\]
Thus \(0\le\psi\le1\), with \(\psi=0\) on
\([0,r_1]\cup[r_4,\infty)\) and \(\psi=1\) on
\([r_2,r_3]\). Put \(I_\psi=\int_0^\infty r^3\psi(r)\,dr>0\)
and, with \(\rho=r^2\), define
\begin{equation}\label{eq:radial-density}
b_a=\frac{a^4}{4I_\psi},\qquad
\beta_a(\rho)=b_a\psi(\sqrt\rho),\qquad
m_a(\rho)=\int_0^\rho s\beta_a(s)\,ds.
\end{equation}
Changing variables gives \(m_a(\infty)=2b_aI_\psi=a^4/2\),
hence \(0\le m_a\le a^4/2\). Set
\begin{equation}\label{eq:radial-coefficients}
\begin{aligned}
D_a(\rho)&=\rho^2+a^4-2m_a(\rho)\ge\rho^2,
& A_a^{\mathrm{mod}}&=\frac{\sqrt{D_a}}{2\rho},\\
B_a^{\mathrm{mod}}&=A_a^{\mathrm{mod}}
+\rho(A_a^{\mathrm{mod}})'
=\frac{\rho(1-\beta_a)}{2\sqrt{D_a}},
&4A_a^{\mathrm{mod}}B_a^{\mathrm{mod}}&=1-\beta_a.
\end{aligned}
\end{equation}
The expression for \(B_a^{\mathrm{mod}}\) follows from
\(D_a'=2\rho(1-\beta_a)\). Define the potential
\[
\Phi_a^{\mathrm{mod}}(\rho)
=\frac{r_4^2}{2}+\int_{r_4^2}^{\rho}A_a^{\mathrm{mod}}(s)\,ds,
\]
and let \(\omega_a^{\mathrm{mod}}=\ii\partial\bar\partial
\Phi_a^{\mathrm{mod}}\). Its Hessian eigenvalues are
\(A_a^{\mathrm{mod}}\) and \(B_a^{\mathrm{mod}}\).
For \(a^4<4I_\psi\), the inequality \(0\le\beta_a\le b_a<1\)
therefore makes the form positive.

On \(r\le r_1\), one has \(\beta_a=m_a=0\), so
\(D_a=\rho^2+a^4\). On \(r\ge r_4\), one has
\(m_a=a^4/2\), so \(D_a=\rho^2\) and
\(A_a^{\mathrm{mod}}=B_a^{\mathrm{mod}}=1/2\).
Consequently,
\begin{equation}\label{eq:model-matching}
\omega_a^{\mathrm{mod}}=\omega_{\mathrm{EH},a}\quad(r\le r_1),
\qquad \omega_a^{\mathrm{mod}}=\omega_0\quad(r\ge r_4).
\end{equation}
The form extends smoothly across the exceptional curve, and the
chosen additive constant gives \(\Phi_a^{\mathrm{mod}}=\rho/2\)
on the outer region.

\subsection{The global explicit Kummer background}
\label{subsec:global-background}
For \(a^4<4I_\psi\), use the radial construction in the sixteen
disjoint neighborhoods and the flat form elsewhere. Exact matching
on \(r\ge r_4\) gives a smooth global K\"ahler form
\(\omega_a^{\mathrm{mod}}\). Its exceptional areas are
\(\int_{E_\nu}\omega_a^{\mathrm{mod}}=\pi a^2\) by
\eqref{eq:EH-exceptional-area}. Since the difference from the
flat class is supported near the exceptional curves, it is a
linear combination of their classes. Pairing with the curves and
using \eqref{eq:intersection-normalizations} gives
\begin{equation}\label{eq:mod-class}
[\omega_a^{\mathrm{mod}}]
=\kappa-\frac{\pi a^2}{2}\sum_{\nu=1}^{16}[E_\nu]=\alpha_a.
\end{equation}
Thus the background proves that \(\alpha_a\) is K\"ahler
throughout this construction range.

Extend \(\beta_a\) by zero outside the sixteen transition annuli.
The Eguchi--Hanson volume identity and the radial coefficients give
\begin{equation}\label{eq:global-volume-density}
\frac{(\omega_a^{\mathrm{mod}})^2}{2}=(1-\beta_a)\mu.
\end{equation}
On each annulus the quotient radial measure is \(\pi^2r^3\,dr\),
so \(\int_{\text{one annulus}}\beta_a\mu
=\pi^2b_aI_\psi=\pi^2a^4/4\). Summing the sixteen contributions,
\[
\int_X\beta_a\mu=4\pi^2a^4,
\qquad \int_X\frac{(\omega_a^{\mathrm{mod}})^2}{2}
=\frac12-4\pi^2a^4=\frac{\alpha_a^2}{2},
\]
in agreement with \eqref{eq:alpha-volume}.

\subsection{Scale-adapted H\"older norms}
\label{subsec:scale-adapted-norms}
Let \(g_a^{\mathrm{mod}}\) be the metric of
\(\omega_a^{\mathrm{mod}}\). In each exceptional neighborhood,
let \(r_\nu\) be the orbifold radius, extended by zero on
\(E_\nu\). Choose \(s_a=(a^4+r_\nu^4)^{1/4}\) near the
exceptional curve and interpolate to \(s_a=1\) in a fixed outer
collar, using the same fixed cutoff in all sixteen neighborhoods.
The function is smooth on the resolution, since
\(r_\nu^4=|t|^2(1+|\zeta|^2)^2\) in a resolution chart.
It satisfies \(c(a+r_\nu)\le s_a\le C(a+r_\nu)\), with
constants independent of small \(a\).
Lemma~\ref{lem:uniform-adapted-charts-final} also gives
\begin{equation}\label{eq:uniform-scaled-injectivity}
\operatorname{inj}_{g_a^{\mathrm{mod}}}(x)
\ge c_{\mathrm{inj}}s_a(x).
\end{equation}
Fix one \(0<c_*<c_{\mathrm{inj}}/2\), independent of \(a\),
small enough that \(s_a\) is uniformly comparable to \(s_a(x)\)
on \(B_{c_*s_a(x)}(x)\).

For an integer \(k\ge0\) and \(0<\alpha<1\), write
\(d_a=d_{g_a^{\mathrm{mod}}}\) and let \(\nabla\) denote the
induced tensor connection. Set
\[
\|T\|_{C_{\mathrm{sc}}^k}
=\sum_{j=0}^k\sup_Xs_a^j|\nabla^jT|_{g_a^{\mathrm{mod}}},
\qquad \|T\|_{C_{\mathrm{sc}}^{k,\alpha}}
=\|T\|_{C_{\mathrm{sc}}^k}+[T]_{C_{\mathrm{sc}}^{k,\alpha}}.
\]
With \(s_a(x,y)=\min\{s_a(x),s_a(y)\}\), the seminorm is
\[
[T]_{C_{\mathrm{sc}}^{k,\alpha}}
=\sup_{\substack{x\ne y\\d_a(x,y)\le c_*s_a(x,y)}}
s_a(x,y)^{k+\alpha}
\frac{|\nabla^kT(x)-\mathsf P_{yx}\nabla^kT(y)|
_{g_a^{\mathrm{mod}}(x)}}{d_a(x,y)^\alpha}.
\]
Here \(\mathsf P_{yx}\) is parallel transport along the unique
minimizing geodesic. For Hermitian endomorphisms we use the
Frobenius norm. On fixed physical annuli these norms are uniformly
equivalent to ordinary H\"older norms. On the core, the dilation
\(S_a(y)=ay\) gives
\(a^{-2}S_a^*g_a^{\mathrm{mod}}=g_{\mathrm{EH},1}\), and
\(s_a\circ S_a/a\) is uniformly comparable to one on bounded
regions. The scaled norms of scalar functions and Hermitian
endomorphisms are therefore uniformly equivalent to ordinary norms
in these coordinates. For a tensor with \(p\) contravariant and
\(q\) covariant indices, the comparison includes \(a^{p-q}\).

\subsection{The normalized Calabi--Yau residual}
\label{subsec:model-residual}
Let \(dV_a=(\omega_a^{\mathrm{mod}})^2/2\). The target volume
form is \(c_a\mu\), so the normalized residual is
\begin{equation}\label{eq:fa-explicit}
f_a:=\frac{c_a\mu}{dV_a}-1
=\frac{c_a}{1-\beta_a}-1
=\frac{\beta_a-8\pi^2a^4}{1-\beta_a}.
\end{equation}
The difference \(f_a+8\pi^2a^4\) is supported in the sixteen fixed transition annuli.

\begin{proposition}[Residual of the explicit Kummer background]
\label{prop:model-residual}
For every fixed integer \(k\ge0\) and \(0<\alpha<1\), there
are constants \(C_{k,\alpha},a_{k,\alpha}>0\) such that, for
\(0<a<a_{k,\alpha}\),
\begin{equation}\label{eq:fa-bound}
\|f_a\|_{C_{\mathrm{sc}}^{k,\alpha}}\le C_{k,\alpha}a^4,
\qquad \int_Xf_a\,dV_a=0.
\end{equation}
\end{proposition}
\begin{proof}
Since \(b_a=a^4/(4I_\psi)\) and the cutoff is fixed,
\(\|\beta_a\|_{C_{\mathrm{sc}}^{k,\alpha}}\le C_{k,\alpha}a^4\).
Indeed, its support lies in fixed physical annuli, where the
scaled and ordinary norms are uniformly equivalent. For small
\(a\), one has \(\|\beta_a\|_{C^0}<1/2\); all derivatives
of \(s\mapsto(1-s)^{-1}\) are then uniformly bounded on its
range. Formula~\eqref{eq:fa-explicit} yields
\[
\|f_a\|_{C_{\mathrm{sc}}^{k,\alpha}}
\le C_{k,\alpha}\bigl(\|\beta_a\|_{C_{\mathrm{sc}}^{k,\alpha}}
+a^4\bigr)\le C_{k,\alpha}a^4.
\]
The mean-zero identity follows from
\(\int_Xf_a\,dV_a=c_a\int_X\mu-\int_XdV_a=0\), using
\eqref{eq:alpha-volume} and \eqref{eq:global-volume-density}.
\end{proof}

\begin{remark}[The role of the parameter range]
\label{rem:a4-sufficient}
The bound \(a^4<4I_\psi\) ensures that this background is
K\"ahler. Convergence of the correction series is a separate
conclusion of the estimates
\(\|G_af_a\|_{\mathcal X_a}=O(a^3)\) and
\(\|G_aQ_a\|_{\mathrm{bil}}=O(a^{-1})\), whose product is
\(O(a^2)\). The resulting smallness condition is sufficient;
it does not characterize the K\"ahler cone or the maximal
convergence range of the series. In fact, on \(\psi=1\),
the residual is \(b_a(1-32\pi^2I_\psi)/(1-b_a)\), which
can be large while the background remains positive.
\end{remark}

\subsection{The periodic Green kernel}
\label{subsec:periodic-green}
Using real coordinates \(x=(x_1,\ldots,x_4)\) on \(T=\mathbb R^4/\mathbb Z^4\),
let \(\Delta_0=\sum_{j=1}^4\partial_{x_j}^2\), so that
\(-\Delta_0\) is nonnegative. The normalized torus Green
function is characterized by \(-\Delta_0G_T=\delta_0-1\)
and \(\int_TG_T\,dV_0=0\). Distributionally,
\[
G_T(x)=\sum_{n\in\mathbb Z^4\setminus\{0\}}
\frac{e^{2\pi\ii n\cdot x}}{4\pi^2|n|^2}.
\]
For differentiation and computation we use the absolutely
convergent Ewald representation of Appendix~\ref{app:ewald}.
The normalized quotient kernel is
\begin{equation}\label{eq:GY-general}
G_Y(x,y)=G_T(\widetilde x-\widetilde y)
+G_T(\widetilde x+\widetilde y),
\end{equation}
independent of the choices of lifts. At a two-torsion point,
\(G_Y(x,p_\nu)=2G_T(\widetilde x-\widetilde p_\nu)\), with
\(-\Delta_0G_Y(\,\cdot\,,p_\nu)=\delta_{p_\nu}-2\) and
\(\int_YG_Y(x,p_\nu)\,dV_0(x)=0\). Near \(p_\nu\),
it equals \(1/(2\pi^2r_\nu^2)+H_\nu(x)\), where \(H_\nu\)
is smooth in the orbifold covering coordinates. This kernel and its derivatives provide the exterior
volume and boundary potentials in Section~\ref{sec:green-construction}.

\section{Explicit construction of the global Green operator}
\label{sec:green-construction}

We now construct explicitly the normalized Green operator of the
background metric introduced in Section~\ref{sec:kummer-geometry}.
Throughout this section, \(a>0\) is fixed and sufficiently small.
Write \(\omega_a=\omega_a^{\mathrm{mod}}\) and
\(g_a=g_a^{\mathrm{mod}}\), and set
\(A_a=-\Lambda_{\omega_a}\ddc=-\tfrac12\Delta_{g_a}\).
The interface jump uses the ordinary outward normal derivative,
so the divergence identities carry the corresponding factor two.
On the flat region \(A_a=A_0=-\tfrac12\Delta_0\).
Our aim is to construct \(G_a:\mathcal Y_{a,0}\to\mathcal X_a\)
such that, for every source with \(\int_Xf\,dV_a=0\),
\[
A_aG_af=f,\qquad \int_XG_af\,dV_a=0.
\]
The construction is by domain decomposition.  The exterior problem is
solved using the periodic Green kernel of
Section~\ref{subsec:periodic-green}; the cap problems are solved by
the separated radial kernels of Appendix~\ref{app:EH-kernels}.  The
pieces are then joined by an explicit Dirichlet-to-Neumann matching
operator.

\subsection{Domain decomposition and the exterior particular solution}
\label{subsec:domain-decomposition}

Fix the radius \(R\) from
\eqref{eq:fixed-radii}.  {Since \(R>r_4\), the metric} \(\omega_a\) is exactly flat in a fixed neighborhood of
each hypersurface \(\Sigma_\nu=\{r_\nu=R\}/\{\pm1\}\simeq S^3/\{\pm1\}\).
Write \(\Sigma=\bigsqcup_{\nu=1}^{16}\Sigma_\nu\),
\(X_{\mathrm{out}}=Y\setminus\bigcup_{\nu=1}^{16}B_R(p_\nu)\),
and let \(X_{\nu,\mathrm{cap}}\) be the resolved region bounded by
\(\Sigma_\nu\). Then
\begin{equation}\label{eq:X-domain-decomposition}
X
=
X_{\mathrm{out}}
\cup_{\Sigma}
\bigcup_{\nu=1}^{16}
X_{\nu,\mathrm{cap}}.
\end{equation}
For a piecewise smooth function \(u = \left( u_{\mathrm{out}}, u_1,\ldots,u_{16} \right)\) whose traces agree on \(\Sigma\), define its conormal jump by
\[
(\mathcal J_a u)_\nu
=
\partial_{\nu_{\mathrm{out}}}
u_{\mathrm{out}}
+
\partial_{\nu_{\mathrm{cap}}}
u_\nu.
\]
Here both normal vectors point outward from their respective domains.
Because the metric is flat in a neighborhood of the interface, these
are the ordinary Euclidean normal derivatives there.

A continuous piecewise \(A_a\)-harmonic function is a global weak
solution precisely when \(\mathcal J_au=0\).

Let \(G_Y\) be the normalized orbifold Green kernel of
\eqref{eq:GY-general}. Its defining equation gives
\begin{equation}\label{eq:A0-GY}
A_0G_Y(\,\cdot\,,y)
=
\frac12\delta_y-1.
\end{equation}
Fix one of the deleted orbifold points, say \(p_1\), and define
\[
K_{\mathrm{out}}(x,y)=2G_Y(x,y)-2G_Y(x,p_1),\qquad
(\mathcal P_{\mathrm{out}}f)(x)=
\int_{X_{\mathrm{out}}}K_{\mathrm{out}}(x,y)f(y)\,dV_0(y).
\]
The constant terms in \eqref{eq:A0-GY} cancel, and the remaining
mass at \(p_1\) lies outside \(X_{\mathrm{out}}\). Hence
\(A_0^xK_{\mathrm{out}}(x,y)=\delta_y\) and
\(A_0\mathcal P_{\mathrm{out}}f=f\) on \(X_{\mathrm{out}}\).
The function \(\mathcal P_{\mathrm{out}}f\) does not in general
vanish on the boundary.  We next construct an explicit harmonic
extension operator to remove its trace.

\subsection{Periodic multipoles and exterior harmonic functions}
\label{subsec:outer-harmonic}

Choose a real orthonormal basis \(Y_{\ell,\gamma}\) of spherical
harmonics on \(S^3/\{\pm1\}\), with even degrees
\(\ell=0,2,4,\ldots\). For \(\ell\ge2\), let
\(G_{\nu,\ell,\gamma}\) be the normalized periodic multipole kernel
of Appendix~\ref{app:ewald}, so that near \(p_\nu\),
\[
G_{\nu,\ell,\gamma}
=
r_\nu^{-\ell-2}
Y_{\ell,\gamma}
+
O(1).
\]
Since derivatives kill the constant term in the Green equation, \(A_0G_{\nu,\ell,\gamma}=0\) on \(X_{\mathrm{out}}\).

For the monopole sector, the individual functions \(G_Y(\cdot,p_\nu)\)
are not harmonic on \(X_{\mathrm{out}}\), because of the constant
term in \eqref{eq:A0-GY}.  Their differences are harmonic.  Define
the coefficient space
\[
\mathfrak C
=
\mathbb R
\oplus
\left\{
(c_{\nu,0})_{\nu=1}^{16}:
\sum_{\nu=1}^{16}c_{\nu,0}=0
\right\}
\oplus
\prod_{\substack{
1\le\nu\le16\\
\ell\ge2,\ \ell\ {\rm even}\\
\gamma
}}
\mathbb R.
\]
The first factor, with coefficient \(c_{\mathrm{const}}\),
represents the global constant harmonic function.
For \(k\ge0\), let \(\mathfrak C^{k,\alpha}\) consist of the
sequences extracted from \(C^{k,\alpha}(\Sigma)\). Their boundary expansion is
\[
(\mathcal I c)|_{\Sigma_\nu}
=c_{\mathrm{const}}+c_{\nu,0}+\sum_{\ell\ge2,\,\ell\ \mathrm{even},\gamma}
c_{\nu,\ell,\gamma}Y_{\ell,\gamma},
\]
where the expansion is understood in \(L^2\), or by Poisson
regularization. The map \(\mathcal I\) recovers the boundary function,
and the coefficient norm is \(\|\mathcal I c\|_{C^{k,\alpha}}\).
The zero-sum condition uniquely separates the componentwise constants from the global constant.

For \(c=\left(c_{\mathrm{const}},c_{\nu,0},c_{\nu,\ell,\gamma}\right)
\in\mathfrak C^{k,\alpha}\), set
\begin{equation}\label{eq:Mout}
\begin{aligned}
\mathcal M_{\mathrm{out}}c
={}&
c_{\mathrm{const}}
+
\sum_{\nu=1}^{16}
c_{\nu,0}
\,2\pi^2R^2
G_Y(\,\cdot\,,p_\nu)\\
&+
\sum_{\substack{
1\le\nu\le16,\ \ell\ge2\\
\ell\ {\rm even},\gamma
}}
c_{\nu,\ell,\gamma}
R^{\ell+2}
G_{\nu,\ell,\gamma}.
\end{aligned}
\end{equation}
The relation \(\sum_\nu c_{\nu,0}=0\) makes every function in
the image of the multipole operator \(A_0\)-harmonic on \(X_{\mathrm{out}}\).

The factors in \eqref{eq:Mout} are chosen so that the singular self-interaction on \(\Sigma_\nu\) has trace
\(Y_{\ell,\gamma}\). For the monopole, the singular part of
\(2\pi^2R^2G_Y\) has trace one on its own boundary component;
the remaining part is smooth.

Let \(\gamma_\Sigma\) be the boundary trace operator and set
\(\mathcal T_{\mathrm{out}}=\gamma_\Sigma\mathcal M_{\mathrm{out}}\).
In the spherical basis,
\(\mathcal T_{\mathrm{out}}\mathcal I^{-1}=I+\mathcal K_{\mathrm{out}}\),
where \(\mathcal K_{\mathrm{out}}\) has a smooth kernel in the boundary
variables, including the finite-rank monopole correction. Indeed,
after removal of the prescribed self-singularity, the regular
multipole columns have geometric decay. 
Here is an explicit estimate. Fix a source radius
\(\varrho=1/4\), so that \(R<\varrho<1/2-R\).
For \(x\in\Sigma_\mu\), set
\(\mathscr F_{\mu\nu}(x,y)=G_Y(x,p_\nu+y)\) if
\(\mu\ne\nu\). If \(\mu=\nu\), in local covering
coordinates \(\xi=x-p_\nu\) define
\[
\mathscr F_{\nu\nu}(x,y)
=G_Y(x,p_\nu+y)
-\frac1{4\pi^2|\xi-y|^2}
-\frac1{4\pi^2|\xi+y|^2},
\]
using the smooth extension of the difference at its removable
singularities. These kernels are smooth for \(|y|\le\varrho\)
and for \(x\) in a fixed neighborhood of the corresponding
boundary component. They satisfy \(\Delta_y\mathscr F_{\mu\nu}=2\),
so subtracting \(|y|^2/4\) makes them harmonic in \(y\).
For every positive even degree, harmonic projection on the unit
link, with quotient measure of volume \(\pi^2\), gives
\[
P_{\ell,\gamma}(\partial_y)\mathscr F_{\mu\nu}(x,0)
=\frac{2^\ell(\ell+1)!}{\pi^2\varrho^\ell}
\int_{S^3/\{\pm1\}}
\mathscr F_{\mu\nu}(x,\varrho\theta)
Y_{\ell,\gamma}(\theta)\,d\theta.
\]
The quadratic particular solution is orthogonal to every positive
harmonic degree. With \(\theta_\nu(x)=(x-p_\nu)/R\) on
\(\Sigma_\nu\), the regular normalized column is therefore
\[
\begin{aligned}
R^{\ell+2}G_{\nu,\ell,\gamma}(x)
&-\mathbf1_{\{\mu=\nu\}}Y_{\ell,\gamma}(\theta_\nu(x))\\
&=2(\ell+1)R^2(R/\varrho)^\ell
\int_{S^3/\{\pm1\}}
\mathscr F_{\mu\nu}(x,\varrho\theta)
Y_{\ell,\gamma}(\theta)\,d\theta.
\end{aligned}
\]
The singular term is omitted when \(\mu\ne\nu\).
Every fixed output derivative is bounded by
\(C_j(1+\ell)(R/\varrho)^\ell\). For normal derivatives the
singular column is subtracted before restricting to the boundary.
Polynomial bounds for derivatives of spherical harmonics and for their
multiplicities give absolute convergence of the regular boundary
kernel with all derivatives in both variables; the constant and
monopole corrections have finite rank.
The singular
self-series is the exterior Poisson extension of its boundary
function: for \(r>R\), its degree-\(\ell\) multiplier is
\((R/r)^{\ell+2}\). It therefore has the prescribed
\(C^{k,\alpha}\) trace. This interpretation does not require
sharp spherical partial sums to converge in the same H\"older norm.

\begin{lemma}[Exterior trace isomorphism]
\label{lem:outer-trace-isomorphism}
For \(k\ge0\), the trace map
\(\mathcal T_{\mathrm{out}}:\mathfrak C^{k,\alpha}
\to C^{k,\alpha}(\Sigma)\) is an isomorphism.
\end{lemma}

\begin{proof}
Since \(\mathcal T_{\mathrm{out}}\mathcal I^{-1}=I+\mathcal K_{\mathrm{out}}\) with \(\mathcal K_{\mathrm{out}}\) smoothing, it is Fredholm of index
zero.

{Suppose \(\mathcal T_{\mathrm{out}}c=0\).} Then \(u=\mathcal M_{\mathrm{out}}c\) is harmonic on \(X_{\mathrm{out}}\) and vanishes on its boundary.
The maximum principle gives \(u=0\) on \(X_{\mathrm{out}}\).

Fix \(\nu\) and a concentric annulus \(R<r_\nu<R+\epsilon\)
containing no other deleted ball. The singular self-multipoles converge
there to a sum of \(r_\nu^{-\ell-2}Y_{\ell,\gamma}\). The regular
self-parts and the other-center contributions extend smoothly
across the ball. Their sum is harmonic: the higher multipoles are
harmonic, and the constant monopole Laplacians cancel because
\(\sum_\nu c_{\nu,0}=0\). Uniqueness of the spherical Laurent
expansion of a harmonic function on this annulus therefore forces
all negative-power coefficients to vanish. This argument does not
require continuation of the singular series to \(r_\nu<R\). Hence
\(c_{\nu,\ell,\gamma}=0, \qquad c_{\nu,0}=0.\) The remaining constant coefficient also vanishes because the boundary
trace is zero.  Thus \(c=0\).  Hence
\(\mathcal T_{\mathrm{out}}\) is injective and therefore invertible.
\end{proof}

\begin{lemma}[Spherical tails of a smoothing operator]
\label{lem:spherical-smoothing-tails}
Let \(K\) have a smooth kernel on the fixed boundary \(\Sigma\).
Let \(P_L\) project onto spherical degrees at most \(L\), and
put \(Q_L=I-P_L\).
For every integer \(k\ge0\) and \(N>0\),
\begin{equation}\label{eq:spherical-smoothing-tail}
\|Q_LKQ_L\|_{C^{k,\alpha}(\Sigma)\to C^{k,\alpha}(\Sigma)}
\le C_{k,\alpha,N}(1+L)^{-N}.
\end{equation}
The constants are uniform for a family with uniformly bounded
kernel derivatives of every fixed order.
\end{lemma}
\begin{proof}
Integration by parts with the spherical Laplacian gives arbitrary
polynomial decay in both indices of the matrix coefficients of
\(K\). The \(C^{k,\alpha}\)-norms of normalized spherical
harmonics grow at most polynomially in their degree, and the
multiplicity of degree \(\ell\) is \((\ell+1)^2\).
The coefficients of an input function are bounded by its
\(L^2\)-norm, hence by its \(C^{k,\alpha}\)-norm. Summing the
{norms of the terms in the output series} over both degrees greater
than \(L\) therefore gives \eqref{eq:spherical-smoothing-tail},
after taking sufficiently many kernel derivatives. This proof
does not require the sharp spherical projections to have uniformly
bounded H\"older operator norms.
\end{proof}

The {exterior trace} inverse has an explicit low/high spherical representation.
Write \(\mathcal T=\mathcal T_{\mathrm{out}}\mathcal I^{-1}
=I+\mathcal K_{\mathrm{out}}\), and let \(P_L\) project onto
degrees at most \(L\), with \(Q_L=I-P_L\); the subscripts
\(LL,LH,HL,HH\) denote the corresponding operator blocks.
By Lemma~\ref{lem:spherical-smoothing-tails}, \(L\) may be
chosen so that
\(\|Q_L\mathcal K_{\mathrm{out}}Q_L\|_{C^{2,\alpha}\to C^{2,\alpha}}<1\).
Then
\begin{equation}\label{eq:outer-high-inverse}
\begin{aligned}
\mathcal T_{HH}^{-1}
&=\sum_{j=0}^{\infty}(-Q_L\mathcal K_{\mathrm{out}}Q_L)^jQ_L,\\
\mathcal S_{\mathrm{out},L}
&=\mathcal T_{LL}-\mathcal T_{LH}\mathcal T_{HH}^{-1}\mathcal T_{HL}.
\end{aligned}
\end{equation}
The finite matrix \(\mathcal S_{\mathrm{out},L}\) is invertible
because \(\mathcal T\) is invertible. Here all low modes, including
the global constant, are retained. The usual block inverse, written
out for the interface operator in \eqref{eq:Na-full-inverse}, gives
\(\mathcal T^{-1}\), and
\(\mathcal T_{\mathrm{out}}^{-1}=\mathcal I^{-1}\mathcal T^{-1}\).
The exterior harmonic extension is therefore
\(H_{\mathrm{out}}\phi=
\mathcal M_{\mathrm{out}}\mathcal T_{\mathrm{out}}^{-1}\phi\).
It satisfies \(A_0H_{\mathrm{out}}\phi=0\) and
\(\gamma_\Sigma H_{\mathrm{out}}\phi=\phi\), so the explicit
zero-boundary exterior Dirichlet inverse is
\[
D_{\mathrm{out}}f
=
\mathcal P_{\mathrm{out}}f
-
H_{\mathrm{out}}
\left(
\gamma_\Sigma
\mathcal P_{\mathrm{out}}f
\right).
\]
Thus \(A_0D_{\mathrm{out}}f=f\) and
\(D_{\mathrm{out}}f|_\Sigma=0\).
\subsection{The cap Dirichlet inverses}
\label{subsec:cap-dirichlet}

On each cap \(X_{\nu,\mathrm{cap}}\), the background metric is radial.
The complete separated construction is given in
Appendix~\ref{app:EH-kernels}; we recall only the formulas needed
here.

Let \(P_{p,q,\gamma}(r\theta)=r^\ell Y_{p,q,\gamma}(\theta)\)
be a bihomogeneous harmonic polynomial of bidegree \((p,q)\),
where \(p,q\) are nonnegative integers and \(\ell=p+q\) is even.
Writing \(u=w(\rho)P_{p,q,\gamma}\), where \(\rho=r^2\), reduces
\(A_au=F(\rho)P_{p,q,\gamma}\) to
\begin{equation}\label{eq:cap-SL}
-
\left(
\rho^{\ell+2}A_a^{\mathrm{mod}}w'
\right)'
+
pq\,\rho^{\ell+1}
(A_a^{\mathrm{mod}})'w
=
\rho^{\ell+1}
A_a^{\mathrm{mod}}
B_a^{\mathrm{mod}}
F.
\end{equation}
Let \(w_{p,q,a}^{\mathrm{reg}}\) be the solution for which \(w_{p,q,a}^{\mathrm{reg}} P_{p,q,\gamma}\) extends smoothly across the exceptional curve.  Define the
zero-boundary second solution by reduction of order:
\[
w_{p,q,a}^{D}(\rho)
=
w_{p,q,a}^{\mathrm{reg}}(\rho)
\int_\rho^{R^2}
\frac{
ds
}{
s^{\ell+2}
A_a^{\mathrm{mod}}(s)
\left(
w_{p,q,a}^{\mathrm{reg}}(s)
\right)^2
}.
\]
The radial Dirichlet kernel is
\(g_{p,q,a}^{D}(\rho,\sigma)=
w_{p,q,a}^{\mathrm{reg}}(\rho_<)w_{p,q,a}^{D}(\rho_>)\).
Expanding the source in the orthonormal angular basis gives
\[
\begin{aligned}
D_{\nu,a}f&=\sum_{p,q,\gamma}
w_{p,q,\gamma}(\rho)P_{p,q,\gamma}(z),\\
w_{p,q,\gamma}(\rho)&=
\int_0^{R^2}g_{p,q,a}^{D}(\rho,\sigma)
\sigma^{\ell+1}A_a^{\mathrm{mod}}(\sigma)B_a^{\mathrm{mod}}(\sigma)
F_{p,q,\gamma}(\sigma)\,d\sigma.
\end{aligned}
\]
Here \(F_{p,q,\gamma}\) denotes the corresponding radial coefficient
of the source in the polynomial normalization of
Appendix~\ref{app:EH-kernels}.

The harmonic extension of a boundary mode and its
Dirichlet-to-Neumann eigenvalue are
\[
H_{\nu,a}Y_{p,q,\gamma}=
\frac{r^\ell w_{p,q,a}^{\mathrm{reg}}(r^2)}
{R^\ell w_{p,q,a}^{\mathrm{reg}}(R^2)}Y_{p,q,\gamma},\qquad
\lambda_{p,q}^{\mathrm{cap}}(a)=\frac{\ell}{R}
+2R\frac{(w_{p,q,a}^{\mathrm{reg}})'(R^2)}
{w_{p,q,a}^{\mathrm{reg}}(R^2)}.
\]
Thus the cap inverses and their boundary operators are completely
determined by explicit one-dimensional radial data.

\subsection{The interface operator and its inverse}
\label{subsec:block-operators}

Combine the Dirichlet inverses and harmonic extensions into
\[
\mathcal D_a=D_{\mathrm{out}}\oplus D_{1,a}\oplus\cdots\oplus D_{16,a},
\qquad
\mathcal H_a\phi=(H_{\mathrm{out}}\phi,
H_{1,a}\phi_1,\ldots,H_{16,a}\phi_{16}),
\]
and define the global interface operator by
\(\mathcal N_a=\mathcal J_a\mathcal H_a\).
For every boundary datum,
Green's identity gives
\begin{equation}\label{eq:Na-energy}
\begin{aligned}
\langle\mathcal N_a\phi,\phi\rangle_\Sigma
={}&
\int_{X_{\mathrm{out}}}
|\nabla H_{\mathrm{out}}\phi|_{g_0}^2\,dV_0\\
&+
\sum_{\nu=1}^{16}
\int_{X_{\nu,\mathrm{cap}}}
|\nabla H_{\nu,a}\phi_\nu|_{g_a}^2\,dV_a.
\end{aligned}
\end{equation}
The bilinear Green identity shows that \(\mathcal N_a\) is
self-adjoint with respect to the physical boundary measure.
Consequently
\begin{equation}\label{eq:Na-kernel}
\ker\mathcal N_a
=
\operatorname{span}\{\mathbf 1\},
\end{equation}
where \(\mathbf 1=(1,\ldots,1)\) is the common constant boundary value.

Let \(\mathcal Z_0\) and \(\mathcal W_0\) be the zero-boundary-mean
subspaces of \(C^{2,\alpha}(\Sigma)\) and \(C^{1,\alpha}(\Sigma)\),
respectively.
The inverse below maps \(\mathcal W_0\) to \(\mathcal Z_0\).
The splitting~\eqref{eq:Na-splitting} makes \(\mathcal N_a\)
Fredholm of index zero, so self-adjointness and
\eqref{eq:Na-kernel} identify its range with \(\mathcal W_0\).
Any fixed normalization of the solution by constants would be equivalent.
Denote the inverse on this complement by \(\mathcal N_a^\dagger\).
To construct it explicitly, first consider a Euclidean sphere of
radius \(R\). The regular interior extension of
\(Y_{\ell,\gamma}\) is \((r/R)^\ell Y_{\ell,\gamma}\), with
outward derivative \((\ell/R)Y_{\ell,\gamma}\); the singular exterior
extension is \((R/r)^{\ell+2}Y_{\ell,\gamma}\), with outward derivative
\(((\ell+2)/R)Y_{\ell,\gamma}\) because the exterior normal is
\(-\partial_r\). The sum of these Dirichlet-to-Neumann eigenvalues is \(2(\ell+1)/R\).
Thus, defining
\(\mathcal A_\Sigma Y_{\ell,\gamma}=((\ell+1)/R)Y_{\ell,\gamma}\)
on each boundary component, we obtain
\begin{equation}\label{eq:Na-splitting}
\mathcal N_a
=
2\mathcal A_\Sigma
+
\mathcal K_a.
\end{equation}
The remainder \(\mathcal K_a\) comprises the regular periodic
self-interactions, interactions between distinct boundary components,
and the reflected cap terms. The periodic terms have smooth boundary
kernels, and the cap terms are smoothing because the metric is exactly
flat on the fixed collar \(r_4<r<R\). Their uniform estimates are
proved in Section~\ref{sec:green-estimates}.

For an even integer \(L\ge0\), let \(P_L\) be the orthogonal projection onto all boundary spherical harmonics
with degree \(0\le\ell\le L\), and put \(Q_L=I-P_L\). Choose once and for all \(L\) sufficiently large that
\begin{equation}\label{eq:high-mode-small}
\left\|
(2\mathcal A_\Sigma)^{-1}
Q_L\mathcal K_aQ_L
\right\|
<q<1
\end{equation}
for every sufficiently small \(a\).  The existence of such an \(L\)
is proved in
Section~\ref{subsec:uniform-high-modes}.

Set \(\mathcal N_{HH}=Q_L\mathcal N_aQ_L\). Its inverse is the absolutely convergent series
\begin{equation}\label{eq:NHH-inverse}
\mathcal N_{HH}^{-1}
=
\sum_{m=0}^{\infty}
\left[
-(2\mathcal A_\Sigma)^{-1}
Q_L\mathcal K_aQ_L
\right]^m
(2\mathcal A_\Sigma)^{-1}Q_L.
\end{equation}
Thus all high boundary modes are inverted by an explicitly prescribed
Neumann series.

Write \(\mathcal N_{LL},\mathcal N_{LH},\mathcal N_{HL},
\mathcal N_{HH}\) for the blocks relative to \(P_L\oplus Q_L\).
The finite-dimensional Schur complement is
\[
\mathcal S_{a,L}
=
\mathcal N_{LL}
-
\mathcal N_{LH}
\mathcal N_{HH}^{-1}
\mathcal N_{HL}.
\]
Restrict the low-mode space to \(P_L\mathcal Z_0\), removing
the common constant direction, and choose a fixed spherical basis \(\{Y_I\}_{I\in\mathcal I_L}\) of the remaining low-mode space, orthonormal for the physical
\(L^2(\Sigma,d\sigma)\) inner product. Define
\begin{equation}\label{eq:M-a-L}
M_{a,L}[I,J]
=
\left\langle
Y_I,
\mathcal S_{a,L}Y_J
\right\rangle_\Sigma.
\end{equation}
Every matrix entry is prescribed by the periodic Ewald kernels,
radial cap solutions, and high-mode series \eqref{eq:NHH-inverse}.

By \eqref{eq:Na-energy}, the full interface operator has no kernel
apart from the global constant.  Hence \(M_{a,L}\) is nonsingular.  
Its inverse is given by the finite-dimensional formula
\begin{equation}\label{eq:M-inverse}
M_{a,L}^{-1}
=
\frac{
\operatorname{adj}M_{a,L}
}{
\det M_{a,L}
}.
\end{equation}
The exterior trace inverse is treated by the same finite low-mode
and convergent high-mode procedure. Thus the construction involves
only finitely many finite-dimensional inversions, together with the
stated convergent series.

On the complement of the common constant, self-adjointness and
\(\mathcal N_a\mathbf1=0\) ensure that the blocks act between the corresponding low- and high-mode parts of the mean-zero spaces.
 The interface inverse is given by the block formula:
\begin{equation}\label{eq:Na-full-inverse}
\mathcal N_a^\dagger
=
\begin{pmatrix}
\mathcal S_{a,L}^{-1}
&
-\mathcal S_{a,L}^{-1}
\mathcal N_{LH}
\mathcal N_{HH}^{-1}
\\[1mm]
-\mathcal N_{HH}^{-1}
\mathcal N_{HL}
\mathcal S_{a,L}^{-1}
&
\mathcal N_{HH}^{-1}
+
\mathcal N_{HH}^{-1}
\mathcal N_{HL}
\mathcal S_{a,L}^{-1}
\mathcal N_{LH}
\mathcal N_{HH}^{-1}
\end{pmatrix}.
\end{equation}
Here \(\mathcal S_{a,L}^{-1}\) is represented by
\eqref{eq:M-inverse} in the fixed low-mode basis.

\subsection{The normalized global inverse}
\label{subsec:global-Schur}

Let \(f\in C^\infty(X)\) satisfy \(\int_Xf\,dV_a=0\).
Restrict \(f\) to the pieces of \eqref{eq:X-domain-decomposition}
and set \(u^{\mathrm D}=\mathcal D_af\). Then
\(A_au^{\mathrm D}=f\) on each piece and
\(u^{\mathrm D}|_\Sigma=0\). Seeking the solution as
\(u=u^{\mathrm D}+\mathcal H_a\phi\) makes the boundary values agree;
the conormal matching condition \(\mathcal J_au=0\) becomes
\[
\mathcal N_a\phi
=
-\mathcal J_a\mathcal D_af.
\]
The right-hand side is orthogonal to the constant mode.  Indeed,
integrating the equations for the zero-boundary solutions and applying
the divergence theorem gives
\[
\left\langle
\mathcal J_a\mathcal D_af,
\mathbf1
\right\rangle_\Sigma
=
-2\int_Xf\,dV_a
=
0.
\]
Thus \(\phi=-\mathcal N_a^\dagger\mathcal J_a\mathcal D_af\),
and the global solution before normalization is
\begin{equation}\label{eq:Gtilde}
\widetilde G_af
=
\mathcal D_af
-
\mathcal H_a
\mathcal N_a^\dagger
\mathcal J_a
\mathcal D_af.
\end{equation}
It satisfies \(A_a\widetilde G_af=f\) globally on \(X\).
Removing the additive constant gives
\begin{equation}\label{eq:Ga-final}
G_af
=
\widetilde G_af
-
\frac{
\displaystyle
\int_X
\widetilde G_af\,dV_a
}{
\displaystyle
\int_XdV_a
}.
\end{equation}
Hence \(A_aG_af=f\) and \(\int_XG_af\,dV_a=0\).
Combining the preceding formulas gives the following explicit
representation.

\begin{theorem}[Explicit global Green representation]
\label{thm:explicit-global-green}
For every sufficiently small \(a>0\), the normalized inverse of \(A_a = -\Lambda_{\omega_a^{\mathrm{mod}}}\ddc\) on the zero-mean space is given by
\eqref{eq:Ga-final}.

More precisely, the construction uses only:

\begin{enumerate}
\item
the periodic Fourier--Ewald Green kernel and its derivatives;

\item
periodic multipole kernels and spherical-harmonic projection;

\item
one-dimensional radial Green kernels on the sixteen Kummer caps;

\item
the high-mode Neumann series for the exterior trace and interface
operators, \eqref{eq:outer-high-inverse} and \eqref{eq:NHH-inverse};

\item
their finite-dimensional Schur inverses, including
\eqref{eq:M-inverse};

\item
differentiation, volume and boundary integration, and finite-dimensional
algebraic operations, together with the convergent series above.
\end{enumerate}

For every smooth source with \(\int_Xf\,dV_a=0\), the resulting
function satisfies \(A_aG_af=f\) and \(\int_XG_af\,dV_a=0\).
\end{theorem}

\begin{remark}[No numerical certification is needed]
\label{rem:no-numerical-Green}
The definition of \(G_a\) does not require a numerical lower bound for
the smallest singular value of \(M_{a,L}\), nor a numerically optimized
high-mode contraction constant.  The next section proves
qualitatively that the low-mode inverse remains uniformly bounded and
that one fixed sufficiently large \(L\) makes
\eqref{eq:high-mode-small} hold uniformly for all sufficiently small
\(a\).
\end{remark}

The local formulas used in this construction are derived in
Appendices~\ref{app:ewald} and \ref{app:EH-kernels}.

\section{Uniform estimates for the global Green operator}
\label{sec:green-estimates}

We first estimate the Green operator in the scale-weighted potential
norm, then obtain the bound \(Ca^{-1}\) in the unweighted
complex-Hessian norm required by the Monge--Amp\`ere equation.
These norms are not uniformly equivalent as \(a\to0\). 
In this section, write \(\omega_a=\omega_a^{\mathrm{mod}}\) and \(g_a=g_a^{\mathrm{mod}}\). The
interfaces lie in fixed flat collars; away from the exceptional
curves the backgrounds converge smoothly to the flat metric, while
on the core scale they are homothetic to the fixed Eguchi--Hanson
metric.

\subsection{The analytic spaces}
\label{subsec:analytic-spaces}

Fix \(0<\alpha<1\), and let \(\mathcal Y_a=C_{\mathrm{sc}}^{0,\alpha}(X)\)
with the norm of Section~\ref{subsec:scale-adapted-norms}. Its
zero-mean subspace is
\(\mathcal Y_{a,0}=\{f\in\mathcal Y_a:\int_Xf\,dV_a=0\}\).
For a potential \(u\), write
\(H_{a,u}=(\omega_a^{\mathrm{mod}})^{-1}\ddc u\), and equip
\(\mathcal X_a=C_{\mathrm{sc}}^{2,\alpha}(X)/\mathbb R\) with
\begin{equation}\label{eq:Xa-norm}
\|u\|_{\mathcal X_a}=\|H_{a,u}\|_{C_{\mathrm{sc}}^{0,\alpha}}.
\end{equation}
For fixed \(a>0\), this is a Banach norm equivalent to the ordinary
\(C^{2,\alpha}\)-quotient norm: use
\(A_au=-\operatorname{tr}H_{a,u}\) and fixed-metric elliptic
estimates on the mean-zero subspace. This equivalence may depend on
\(a\). In particular, the factor \(s_a^2\) multiplying second
derivatives in \(C_{\mathrm{sc}}^{2,\alpha}\) cannot be dropped in
a uniform estimate. Indeed, a nonzero compactly supported unit-core
test function rescaled as \(u_a(x)=a^2v(x/a)\) satisfies
\(\|u_a\|_{C_{\mathrm{sc}}^{2,\alpha}}=O(a^2)\), whereas
\(\|u_a\|_{\mathcal X_a}\asymp1\) when its complex Hessian is
nonzero. The test function is understood through the Eguchi--Hanson
dilation and extended by zero; subtracting its average does not
change this comparison.

On the fixed interface
\(\Sigma=\bigsqcup_{\nu=1}^{16}\Sigma_\nu\), we use the
\(a\)-independent spaces \(\mathcal Z=C^{2,\alpha}(\Sigma)\) and
\(\mathcal W=C^{1,\alpha}(\Sigma)\).

\subsection{Uniform geometry in adapted charts}
\label{subsec:uniform-adapted-geometry}

The following geometric bounds give uniform local estimates.

\begin{lemma}[Uniform adapted charts]
\label{lem:uniform-adapted-charts-final}
There is a covering of \(X\), with uniformly bounded overlap, such
that the following hold uniformly for sufficiently small \(a\).
\begin{enumerate}
\item
On the outer region and the fixed transition annuli, charts have
fixed physical size, and \(g_a^{\mathrm{mod}}\) is uniformly
equivalent in every \(C^k\)-norm to the Euclidean metric.
\item
On the core \(r_\nu\le Ca\), the change of variables \(x=ay\)
identifies \(a^{-2}g_a^{\mathrm{mod}}\) with the unit-scale
Eguchi--Hanson metric on fixed \(y\)-regions.
\item
On \(Ca\le r_\nu\le r_4\), use dyadic annuli
\(\lambda\le r_\nu\le2\lambda\). After rescaling
\(x=\lambda y\), the metric coefficients, their inverses and all
fixed-order derivatives are uniformly bounded, with uniform
ellipticity constants.
\item
The injectivity radius satisfies \eqref{eq:uniform-scaled-injectivity},
and \(s_a\) is uniformly comparable on balls of radius
\(c_*s_a(x)\).
\end{enumerate}
\end{lemma}
\begin{proof}
The first assertion follows from \eqref{eq:model-matching} and the fixed
transition radii. On the core the background is exactly
Eguchi--Hanson by \eqref{eq:model-matching}; its homothety \(x=ay\)
proves the second assertion. On the intermediate region, put
\(\rho=r^2\). The bound \(0\le m_a(\rho)\le a^4/2\) gives
\(D_a(\rho)=\rho^2(1+O(a^4/r^4))\). On a dyadic annulus
\(r\asymp\lambda\), rescaling by \(x=\lambda y\), together with the fixed-cutoff derivative bounds for \(m_a\),
bounds every scale-invariant coefficient derivative by a universal
function of \(a/\lambda\le C\). Uniform ellipticity follows from
\[
A_a^{\mathrm{mod}}=\frac{\sqrt{D_a}}{2\rho},
\qquad
B_a^{\mathrm{mod}}=\frac{\rho(1-\beta_a)}{2\sqrt{D_a}}.
\]
For the last assertion, the rescaled core is a fixed bounded region
of the unit Eguchi--Hanson space, and the rescaled dyadic annuli
form a smooth compact family of noncollapsed metrics including the
flat limit. Together with the fixed outer charts, they have a common
positive injectivity-radius lower bound; scaling back proves
\eqref{eq:uniform-scaled-injectivity}. Enlarge the core charts and
dyadic annuli before restricting to smaller subcharts, so that the
scale balls remain in regions of uniform geometric control. The
explicit scale profile has bounded scale-invariant derivatives;
decreasing the fixed chart fraction then gives local comparability
of \(s_a\).
\end{proof}
Thus the interior and boundary Schauder constants for \(A_a\) are
uniform in the scale-adapted norms.

\subsection{Uniform Dirichlet estimates on the caps}
\label{subsec:uniform-cap-estimates}

For the explicit cap inverse \(D_{\nu,a}\) of
Section~\ref{subsec:cap-dirichlet}, a uniform supremum estimate
comes from a radial barrier. Let \(\tau_a(\rho)\) be smooth across
the exceptional curve, satisfy \(A_a\tau_a=1\), and vanish at
\(\rho=R^2\). The {equation for the} mode \(\ell=p=q=0\) in \eqref{eq:cap-SL},
{integrated} using regularity at the center, gives
\[
\begin{aligned}
-\bigl(\rho^2A_a^{\mathrm{mod}}\tau_a'\bigr)'
&=\rho A_a^{\mathrm{mod}}B_a^{\mathrm{mod}},\\
-\rho^2A_a^{\mathrm{mod}}(\rho)\tau_a'(\rho)
&=\int_0^\rho sA_a^{\mathrm{mod}}(s)B_a^{\mathrm{mod}}(s)\,ds.
\end{aligned}
\]
By \eqref{eq:radial-coefficients},
\(A_a^{\mathrm{mod}}B_a^{\mathrm{mod}}=(1-\beta_a)/4\le1/4\)
and \(A_a^{\mathrm{mod}}\ge1/2\). Hence
\(0\le-\tau_a'(\rho)\le(\int_0^\rho s/4\,ds)/(\rho^2/2)=1/4\).
Since \(\tau_a(R^2)=0\),
\begin{equation}\label{eq:torsion-bound}
0\le\tau_a(\rho)\le\frac{R^2}{4}.
\end{equation}
\begin{lemma}[Uniform cap Dirichlet estimate]
\label{lem:uniform-cap-dirichlet-final}
There exists \(C>0\), independent of sufficiently small \(a\)
and of \(\nu\in\{1,\ldots,16\}\), such that
\begin{equation}\label{eq:uniform-cap-dirichlet-final}
\|D_{\nu,a}f\|_{C_{\mathrm{sc}}^{2,\alpha}}
\le C\|f\|_{C_{\mathrm{sc}}^{0,\alpha}}.
\end{equation}
\end{lemma}
\begin{proof}
The maximum principle and \eqref{eq:torsion-bound} yield
\(\|D_{\nu,a}f\|_{C^0}\le(R^2/4)\|f\|_{C^0}\). Apply the
uniform local Schauder estimates of
Lemma~\ref{lem:uniform-adapted-charts-final} on the core, dyadic
annuli and fixed boundary collar. Bounded overlap combines these
local bounds into \eqref{eq:uniform-cap-dirichlet-final}.
{
More explicitly, on an adapted chart of physical size \(s\), the
rescaled equation has right-hand side \(s^2f\). Its interior
Schauder estimate controls the local scale-weighted
\(C^{2,\alpha}\)-norm of the solution by
\(C(\|D_{\nu,a}f\|_{C^0}+s^2\|f\|_{\mathcal Y_a})\).
The chart sizes are uniformly bounded above. Boundary Schauder
estimates apply on the fixed flat collar with zero Dirichlet data.
This proves the stated weighted estimate without removing the
factor \(s^2\) from the second derivatives.}
\end{proof}
The same maximum-principle and Schauder argument for homogeneous
solutions gives the corresponding extension and trace bounds.
\begin{lemma}[Uniform cap harmonic extension]
\label{lem:uniform-cap-harmonic}
For all sufficiently small \(a\),
\[
\begin{aligned}
\|H_{\nu,a}\phi\|_{C_{\mathrm{sc}}^{2,\alpha}}
&\le C\|\phi\|_{C^{2,\alpha}(\Sigma_\nu)},\\
\|\partial_{\nu_{\mathrm{cap}}}H_{\nu,a}\phi\|_{C^{1,\alpha}(\Sigma_\nu)}
&\le C\|\phi\|_{C^{2,\alpha}(\Sigma_\nu)},
\end{aligned}
\]
where \(C\) is independent of \(a\).
\end{lemma}

\subsection{Uniform exterior estimates}
\label{subsec:uniform-outer-estimates}

The exterior domain and its flat metric are independent of \(a\).
Fixed-domain Schauder theory therefore gives
\[
\begin{aligned}
\|D_{\mathrm{out}}f\|_{C^{2,\alpha}(X_{\mathrm{out}})}
&\le C\|f\|_{C^{0,\alpha}(X_{\mathrm{out}})},\\
\|H_{\mathrm{out}}\phi\|_{C^{2,\alpha}(X_{\mathrm{out}})}
&\le C\|\phi\|_{\mathcal Z},\\
\|\partial_{\nu_{\mathrm{out}}}H_{\mathrm{out}}\phi\|_{\mathcal W}
&\le C\|\phi\|_{\mathcal Z}.
\end{aligned}
\]
In the following block estimates, piecewise norms mean the maximum
over the closures of the seventeen pieces; normal derivatives need
not match across \(\Sigma\). Combining the cap and exterior bounds,
and using the fixed-interface trace estimate, yields
\[
\begin{aligned}
\|\mathcal D_af\|_{C_{\mathrm{sc}}^{2,\alpha}}
&\le C\|f\|_{\mathcal Y_a},\\
\|\mathcal H_a\phi\|_{C_{\mathrm{sc}}^{2,\alpha}}
&\le C\|\phi\|_{\mathcal Z},\\
\|\mathcal J_au\|_{\mathcal W}
&\le C\|u\|_{C_{\mathrm{sc}}^{2,\alpha}}
\end{aligned}
\]
for functions smooth up to the fixed interfaces.

\subsection{Uniform smoothing of the interface remainder}
\label{subsec:uniform-high-modes}

Recall that \(\mathcal N_a=2\mathcal A_\Sigma+\mathcal K_a\)
by \eqref{eq:Na-splitting}. Its remainder has three contributions.
Distinct-boundary interactions involve derivatives of the periodic
Green kernel a fixed positive distance from the diagonal, so their
kernels are uniformly smooth. The same-sphere kernel becomes smooth
once the Euclidean singular part is removed, by
Appendix~\ref{app:ewald}.

For the cap correction, the metric differs from the flat {metric} only
inside \(r\le r_4\). Let \(\lambda_4\) be the interior
Dirichlet-to-Neumann eigenvalue of a separated mode at \(r=r_4\).
Its energy identity gives \(\lambda_4\ge0\). Write its radial
amplitude on the flat collar as \(c_+r^\ell+c_-r^{-\ell-2}\).
Matching the logarithmic derivative at \(r_4\) yields
\[
\frac{c_-}{c_+}r_4^{-2\ell-2}
=\frac{\ell-r_4\lambda_4}{\ell+2+r_4\lambda_4},
\qquad
\left|\frac{c_-}{c_+}r_4^{-2\ell-2}\right|\le1.
\]
In particular \(c_+\ne0\). With \(\vartheta_*=r_4/R<1\), this
implies
\[
|\mathfrak r_{p,q}(a)|\le\vartheta_*^{2\ell+2},
\qquad
\left|\lambda_{p,q}^{\mathrm{cap}}(a)-\frac\ell R\right|
\le\frac{2(\ell+1)}R
\frac{\vartheta_*^{2\ell+2}}{1-\vartheta_*^{2\ell+2}}.
\]
Thus the cap correction is uniformly smoothing. For the exterior,
{
put \(\mathcal T=\mathcal T_{\mathrm{out}}\mathcal I^{-1}
=I+\mathcal K_{\mathrm{out}}\), and let
\(\mathcal D_\Sigma Y_{\ell,\gamma}
=((\ell+2)/R)Y_{\ell,\gamma}\).
The source-sphere estimate also applies to the normal trace, giving
\(\partial_{\nu_{\mathrm{out}}}\mathcal M_{\mathrm{out}}
\mathcal I^{-1}=\mathcal D_\Sigma+\mathcal J_{\mathrm{out}}\),
where \(\mathcal J_{\mathrm{out}}\) has a smooth kernel,
including the finite-rank constant correction. Hence
\[
\Lambda_{\mathrm{out}}-\mathcal D_\Sigma
=(\mathcal J_{\mathrm{out}}-
\mathcal D_\Sigma\mathcal K_{\mathrm{out}})\mathcal T^{-1}.
\]
The trace operator is also invertible on \(L^2(\Sigma)\):
it is a compact perturbation of the identity, and any kernel vector
is smooth from \(v=-\mathcal K_{\mathrm{out}}v\), so
Lemma~\ref{lem:outer-trace-isomorphism} gives injectivity.
Moreover,
\[
\mathcal T^{-1}=I-\mathcal K_{\mathrm{out}}
+\mathcal K_{\mathrm{out}}\mathcal T^{-1}\mathcal K_{\mathrm{out}}.
\]
The last term has smoothing factors on both sides of a bounded
\(L^2\) operator. Thus \(\mathcal T^{-1}-I\), and consequently
the exterior Dirichlet-to-Neumann remainder, have smooth kernels
in both boundary variables. These exterior operators are independent
of \(a\).}
We conclude:
\begin{lemma}[Uniform smoothing]
\label{lem:uniform-smoothing-final}
For every \(m\ge0\), the spherical matrix coefficients satisfy
\[
\left|\langle Y_I,\mathcal K_aY_J\rangle_\Sigma\right|
\le C_m(1+\ell_I+\ell_J)^{-m},
\]
where \(C_m\) is independent of sufficiently small \(a\).
\end{lemma}
This decay and Lemma~\ref{lem:spherical-smoothing-tails}, applied to
\((2\mathcal A_\Sigma)^{-1}\mathcal K_a\), show that the high-mode
tails tend to zero in the \(\mathcal Z\to\mathcal Z\) operator
norm, uniformly for \(0<a<a_0\). Choose a fixed sufficiently large even \(L\)
so that
\begin{equation}\label{eq:high-mode-uniform-contraction}
\sup_{0<a<a_0}
\left\|(2\mathcal A_\Sigma)^{-1}Q_L\mathcal K_aQ_L\right\|
\le q<1.
\end{equation}
The series \eqref{eq:NHH-inverse} then gives
\begin{equation}\label{eq:NHH-uniform}
\|\mathcal N_{HH}^{-1}\|_{\mathcal W\to\mathcal Z}\le C
\end{equation}
uniformly for sufficiently small \(a\).

\subsection{Uniform control of the low modes}
\label{subsec:uniform-low-modes}

{The fixed exterior gives a coercive estimate on boundary data
modulo the common constant. The proof below therefore does not
use convergence of the cap operators to their orbifold limits.} Let
\(\Lambda_{\mathrm{out}}\phi=
\partial_{\nu_{\mathrm{out}}}H_{\mathrm{out}}\phi\), and use the
Sobolev norms on \(\Sigma\) with its physical measure.
\begin{lemma}[Coercivity of the fixed exterior]
\label{lem:exterior-coercivity}
Every \(\phi\in\mathcal Z_0\) satisfies
\begin{equation}\label{eq:exterior-coercivity}
\langle\Lambda_{\mathrm{out}}\phi,\phi\rangle_\Sigma
\ge c_0\|\phi\|_{H^{1/2}(\Sigma)}^2,
\end{equation}
for some \(c_0>0\) independent of \(a\).
\end{lemma}
\begin{proof}
Put \(v=H_{\mathrm{out}}\phi\). Its trace has zero boundary mean
on the fixed connected exterior domain, so the Poincar\'e inequality
gives \(\|v\|_{H^1}\le C\|\nabla v\|_{L^2}\). Indeed, failure
would give a sequence with unit \(H^1\)-norm and gradients tending
to zero. Compactness in \(L^2\) and continuity of the boundary
trace mean force the limit to be a constant of zero boundary mean,
hence zero, contradicting the normalization. The fixed-domain
trace theorem now gives
\[
\|\phi\|_{H^{1/2}(\Sigma)}
\le C\|v\|_{H^1(X_{\mathrm{out}})}
\le C\|\nabla v\|_{L^2(X_{\mathrm{out}})}.
\]
Green's identity identifies the squared gradient norm with
\(\langle\Lambda_{\mathrm{out}}\phi,\phi\rangle_\Sigma\),
proving the claim.
\end{proof}
The cap energies are nonnegative. Hence \eqref{eq:Na-energy} and
\eqref{eq:exterior-coercivity} imply the uniform bound
\(\langle\mathcal N_a\phi,\phi\rangle_\Sigma
\ge c_0\|\phi\|_{H^{1/2}(\Sigma)}^2\) for
\(\phi\in\mathcal Z_0\).
\begin{lemma}[Uniform low-mode inverse]
\label{lem:low-mode-uniform-final}
For the fixed cutoff \(L\) in
\eqref{eq:high-mode-uniform-contraction},
\begin{equation}\label{eq:M-uniform}
\|M_{a,L}^{-1}\|\le C_L
\end{equation}
uniformly for sufficiently small \(a\), after removal of the global
constant direction.
\end{lemma}
\begin{proof}
For \(x\in P_L\mathcal Z_0\), put
\(y=-\mathcal N_{HH}^{-1}\mathcal N_{HL}x\) and \(\psi=x+y\).
The inverse exists by \eqref{eq:NHH-uniform}. Since \(y\) contains
only high spherical modes, it has zero boundary mean and is
\(L^2(\Sigma)\)-orthogonal to \(x\); thus \(\psi\in\mathcal Z_0\).
By construction, \(Q_L\mathcal N_a\psi=0\) and
\(P_L\mathcal N_a\psi=\mathcal S_{a,L}x\), so
\begin{equation}\label{eq:Schur-coercivity}
\begin{aligned}
\langle\mathcal S_{a,L}x,x\rangle_\Sigma
&=\langle\mathcal N_a\psi,\psi\rangle_\Sigma\\
&\ge c_0\|\psi\|_{H^{1/2}(\Sigma)}^2
\ge c_1\|x\|_{L^2(\Sigma)}^2.
\end{aligned}
\end{equation}
{The last inequality uses the continuous inclusion
\(H^{1/2}(\Sigma)\subset L^2(\Sigma)\)}
and the orthogonality of \(x\) and \(y\). Thus every eigenvalue of
\(M_{a,L}\), represented in the fixed physical
\(L^2\)-orthonormal basis, has a uniform positive lower bound.
Its inverse is uniformly bounded; equivalence of norms in this
fixed finite-dimensional space proves \eqref{eq:M-uniform}.
\end{proof}
The off-diagonal blocks are uniformly bounded, {since \(2\mathcal A_\Sigma\) is diagonal in the spherical basis
and the remainders are uniformly smoothing.}
Combining \eqref{eq:NHH-uniform}, \eqref{eq:M-uniform} and the
block formula \eqref{eq:Na-full-inverse} gives:
\begin{proposition}[Uniform interface inverse]
\label{prop:Na-uniform}
For sufficiently small \(a\),
\(\|\mathcal N_a^\dagger\|_{\mathcal W_0\to\mathcal Z_0}\le C\),
with \(C\) independent of \(a\).
\end{proposition}

\subsection{The global estimates and their different norms}
\label{subsec:global-Green-estimate}

{
The supremum estimate can be obtained directly from
\eqref{eq:Gtilde}. The cap barriers and the fixed exterior
Dirichlet estimate give
\(\|\mathcal D_af\|_{C^0}\le C\|f\|_{\mathcal Y_a}\).
Schauder estimates on the fixed boundary collars give
\(\|\mathcal J_a\mathcal D_af\|_{\mathcal W}
\le C\|f\|_{\mathcal Y_a}\).
For \(f\in\mathcal Y_{a,0}\), the flux identity in
Section~\ref{subsec:global-Schur} places this jump in
\(\mathcal W_0\).
Proposition~\ref{prop:Na-uniform} therefore bounds
\(\phi=-\mathcal N_a^\dagger\mathcal J_a\mathcal D_af\)
in \(\mathcal Z\) by the same quantity. The maximum principle
on every piece gives
\(\|\mathcal H_a\phi\|_{C^0}\le\|\phi\|_{C^0(\Sigma)}\).
Consequently \(\|\widetilde G_af\|_{C^0}
\le C\|f\|_{\mathcal Y_a}\).
The local scale-weighted Schauder estimates also give
\(\|\widetilde G_af\|_{C_{\mathrm{sc}}^{2,\alpha}}
\le C\|f\|_{\mathcal Y_a}\).
Matching traces and normal derivatives produces a global weak
solution, and elliptic regularity on the fixed flat collars
controls the norms across each interface. Subtracting the average
{increases the supremum norm} by at most a factor two. Thus}
\begin{equation}\label{eq:Green-C0}
\|G_af\|_{C^0}\le C\|f\|_{\mathcal Y_a}.
\end{equation}
These bounds also extend the inverse to \(f\in\mathcal Y_{a,0}\).
For fixed \(a\), choose smooth mean-zero \(f_j\) such that
\(\|f_j\|_{\mathcal Y_a}\le2\|f\|_{\mathcal Y_a}\) and
\(f_j\to f\) in \(C^{0,\alpha'}\), where \(0<\alpha'<\alpha\).
The preceding estimates give uniform \(C^{2,\alpha}\) bounds for
\(G_af_j\), with constants allowed to depend on this fixed \(a\).
Compactness and the equation give a mean-zero \(C^{2,\alpha}\)
limit satisfying \(A_au=f\). Uniqueness makes it independent of
the approximation, and the bounds pass to the limit; smooth density
in the full \(C^{0,\alpha}\) norm is not needed. These estimates
control weighted potential derivatives. The stronger norm
\(\mathcal X_a\) requires the following gradient bound.
\begin{lemma}[Uniform gradient bound]\label{lem:uniform-gradient}
If \(u=G_af\), then
\begin{equation}\label{eq:gradient-bound}
\|\nabla u\|_{C^0}
\le C\bigl(\|u\|_{C^0}+\|f\|_{C^0}\bigr)
\le C\|f\|_{\mathcal Y_a}.
\end{equation}
\end{lemma}
\begin{proof}
The Ricci tensor has a uniform lower bound
\(\operatorname{Ric}(g_a)\ge-Kg_a\): it vanishes in the
Eguchi--Hanson and flat regions, while on the fixed transition
annuli this follows from
\(\operatorname{Ric}(\omega_a)=-\ddc\log(1-\beta_a)\) and
uniform metric equivalence. For the proof only, put
\(P_t=e^{-tA_a}\). Bochner's identity and the maximum principle
give \(|\nabla P_th|^2\le e^{Kt}P_t|\nabla h|^2\). The variance
identity then yields
\[
\begin{aligned}
P_t(h^2)-(P_th)^2
&=\int_0^tP_s|\nabla P_{t-s}h|^2\,ds\\
&\ge\left(\int_0^te^{-Ks}\,ds\right)|\nabla P_th|^2.
\end{aligned}
\]
Positivity of \(P_t\) gives
\(\|\nabla P_th\|_{C^0}\le Ct^{-1/2}\|h\|_{C^0}\) for
\(0<t\le1\), with \(C\) independent of \(a\). The identity
\(u=P_1u+\int_0^1P_tf\,dt\) proves the first inequality in
\eqref{eq:gradient-bound}; \eqref{eq:Green-C0} proves the second.
The heat semigroup is used only for this estimate {and does not enter}
{the} explicit definition of \(G_a\).
\end{proof}
\begin{theorem}[Sufficient Green estimate]\label{thm:uniform-green}
For sufficiently small \(a>0\) and \(f\in\mathcal Y_{a,0}\),
\begin{equation}\label{eq:uniform-Green-main}
\|G_af\|_{\mathcal X_a}\le Ca^{-1}\|f\|_{\mathcal Y_a},
\end{equation}
with \(C\) independent of \(a\).
\end{theorem}
\begin{proof}
{Put \(u=G_af\).} Take a chart centered at \(x\) of physical radius comparable to
\(s=s_a(x)\), and subtract \(u(x)\) before rescaling.
Lemma~\ref{lem:uniform-gradient} gives
\(\|u-u(x)\|_{C^0(B_{cs}(x))}\le Cs\|f\|_{\mathcal Y_a}\).
The Schauder estimate on the rescaled chart, whose metric
coefficients are uniformly bounded, yields
\[
\begin{aligned}
\|H_{a,u}\|_{C^{0,\alpha}_{\mathrm{sc}}(B_{c's}(x))}
&\le C\left(s^{-2}\|u-u(x)\|_{C^0(B_{cs}(x))}
+\|f\|_{\mathcal Y_a}\right)\\
&\le C(1+s^{-1})\|f\|_{\mathcal Y_a}.
\end{aligned}
\]
Since \(s_a\ge ca\), take the supremum over the adapted covering.
Uniform rescaled geometry gives the same control of the
parallel-transport H\"older seminorm.
\end{proof}
\begin{remark}[The estimate actually used]\label{rem:uniform-Green-consequences}
The bound
\(\|G_a\|_{\mathcal Y_{a,0}\to C_{\mathrm{sc}}^{2,\alpha}}\le C\)
is distinct from \eqref{eq:uniform-Green-main}. No uniform bound in
the unweighted Hessian norm is asserted; a qualitative absence of
new constant modes does not establish one. Combining
\eqref{eq:uniform-Green-main} with the residual estimate gives
\(\|G_af_a\|_{\mathcal X_a}=O(a^3)\), as required below.
\end{remark}
\section{The recursive representation of the Calabi--Yau potential}
\label{sec:nonlinear-correction}
We prove that the explicitly prescribed recursion represents the
normalized potential \(v_a\) in \eqref{eq:intro-Yau-target}.
The convergence estimates use only the background residual, the
Green estimate, and the quadratic volume identity. {The Calabi--Yau theorem supplies a smooth solution. Uniqueness
identifies the series limit with that solution and proves {the smoothness of the series limit}.}

\subsection{A convergent quadratic recursion}
\label{subsec:abstract-recursion}
We first record the elementary Banach-space argument, which will
also be used in the elliptic discussion.
\begin{proposition}[Convergent quadratic recursion]
\label{prop:abstract-recursion}
Let \(\mathcal E\) be a real Banach space, \(F\in\mathcal E\),
and \(\mathcal B:\mathcal E\times\mathcal E\to\mathcal E\)
a bounded bilinear map. Set \(\eta=\|F\|_{\mathcal E}\) and
\(b=\|\mathcal B\|_{\mathrm{bil}}\), and suppose \(4\eta b<1\).
Define
\[
V_1=F,\qquad V_n=\sum_{j=1}^{n-1}\mathcal B(V_j,V_{n-j}),
\qquad n\ge2.
\]
Then
\begin{equation}\label{eq:abstract-Catalan-bound}
\|V_n\|_{\mathcal E}
\le\operatorname{Cat}_{n-1}\eta^n b^{n-1},
\end{equation}
with the right-hand side interpreted as \(\eta\) for \(n=1\).
The series \(u=\sum_{n\ge1}V_n\) converges absolutely, satisfies
\(u=F+\mathcal B(u,u)\), and obeys
\begin{equation}\label{eq:abstract-radius}
\|u\|_{\mathcal E}
\le\frac{2\eta}{1+\sqrt{1-4\eta b}}\le2\eta.
\end{equation}
\end{proposition}
\begin{proof}
The Catalan identity
\(\operatorname{Cat}_{n-1}
=\sum_{j=1}^{n-1}\operatorname{Cat}_{j-1}\operatorname{Cat}_{n-j-1}\)
proves \eqref{eq:abstract-Catalan-bound} by induction. The generating
function \(\sum_{m\ge0}\operatorname{Cat}_m t^m
=2/(1+\sqrt{1-4t})\), valid for \(|t|<1/4\), gives absolute
convergence and \eqref{eq:abstract-radius}.
The double series for \(\mathcal B(u,u)\) converges absolutely,
since the sum of its norms is at most
\(b(\sum_n\|V_n\|_{\mathcal E})^2\). Regrouping it by total
degree gives \(\mathcal B(u,u)=\sum_{n\ge2}V_n=u-F\).
\end{proof}

\subsection{The quadratic volume identity}
\label{subsec:quadratic-term}
For the remainder of the section write
\(\omega_a=\omega_a^{\mathrm{mod}}\),
\(dV_a=\omega_a^2/2\), and \(A_a=-\Lambda_{\omega_a}\ddc\).
Recall that \(\mathcal X_a\) is the potential space modulo
constants with norm \eqref{eq:Xa-norm}, and
\(\mathcal Y_{a,0}\) is the zero-mean source space.
The quadratic term is polarized by
\(Q_a(u,v)=\frac{\ddc u\wedge\ddc v}{\omega_a^2}\).
In an \(\omega_a\)-unitary frame, let \(U\) and \(V\) represent
the relative complex Hessians of \(u\) and \(v\), respectively.
Then
\[
Q_a(u,v)=\frac12\bigl((\tr U)(\tr V)-\tr(UV)\bigr)
=\frac12\langle(\tr U)I-U,V\rangle_F.
\]
For \(2\times2\) Hermitian matrices,
\(|(\tr U)I-U|_F=|U|_F\). This gives the pointwise estimate
\(|Q_a(u,v)|\le\frac12|U|_F|V|_F\). Taking differences after
parallel transport gives the corresponding H\"older product bound.
\begin{proposition}[Quadratic estimate and compatibility]
\label{prop:uniform-quadratic}
The map \(Q_a:\mathcal X_a\times\mathcal X_a\to\mathcal Y_{a,0}\)
is symmetric and bilinear, with
\begin{equation}\label{eq:Qa-uniform}
\|Q_a(u,v)\|_{\mathcal Y_a}
\le\frac12\|u\|_{\mathcal X_a}\|v\|_{\mathcal X_a}.
\end{equation}
In particular, every quadratic source is compatible with \(G_a\).
\end{proposition}
\begin{proof}
The product estimate just proved controls the supremum and the
scale-adapted H\"older seminorm. Moreover,
\[
\int_XQ_a(u,v)\,dV_a
=\frac12\int_X\ddc u\wedge\ddc v=0.
\]
For smooth potentials this is Stokes' theorem. Approximation in
\(C^2\) gives the identity for \(C^{2,\alpha}\) potentials.
\end{proof}
The volume expansion and the prescribed volume equation give
\[
\begin{aligned}
\frac{(\omega_a+\ddc u)^2/2}{dV_a}
&=1+\Lambda_{\omega_a}\ddc u+Q_a(u,u),\\
-A_au+Q_a(u,u)&=f_a,
\qquad u=-G_af_a+G_aQ_a(u,u).
\end{aligned}
\]
These identities hold for mean-zero \(C^{2,\alpha}\) potentials;
positivity will follow by continuity along {the path defined by} the convergent series.

\subsection{The explicit coefficients and their convergence}
\label{subsec:binary-tree-series}
Set \(F_a=-G_af_a\),
\(\eta_a=\|F_a\|_{\mathcal X_a}\), and
\(\mathfrak b_a=\|G_aQ_a\|_{\mathrm{bil}}\).
Theorem~\ref{thm:uniform-green} and
Proposition~\ref{prop:model-residual}, together with
\eqref{eq:Qa-uniform}, give
\begin{equation}\label{eq:Fa-small}
\eta_a\le C_1a^3,
\qquad \mathfrak b_a\le C_2a^{-1}.
\end{equation}
After decreasing \(a_0\), we therefore have
\begin{equation}\label{eq:Catalan-smallness}
4\eta_a\mathfrak b_a\le4C_1C_2a^2<1.
\end{equation}
For the coefficients in \eqref{eq:intro-recursion}, the zero-mean
identity ensures that every application of \(G_a\) is defined.
Applying Proposition~\ref{prop:abstract-recursion} with
\(\mathcal E=\mathcal X_a\) and \(\mathcal B=G_aQ_a\) gives
\begin{equation}\label{eq:Catalan-estimate}
\|U_{a,n}\|_{\mathcal X_a}
\le\operatorname{Cat}_{n-1}\eta_a^n\mathfrak b_a^{n-1}
\le\operatorname{Cat}_{n-1}C_1^nC_2^{n-1}a^{2n+1}.
\end{equation}
Consequently,
\begin{equation}\label{eq:ua-majorant}
u_a:=\sum_{n=1}^{\infty}U_{a,n},
\qquad \|u_a\|_{\mathcal X_a}
\le\frac{2\eta_a}{1+\sqrt{1-4\eta_a\mathfrak b_a}}
\le2\eta_a\le Ca^3.
\end{equation}
All representatives are chosen to have zero \(dV_a\)-mean.
For fixed \(a\), the norm equivalence in
Section~\ref{subsec:analytic-spaces} gives {absolute} convergence in
\(C^{2,\alpha}(X)\), and \(\int_Xu_a\,dV_a=0\).

Repeated substitution gives
\(U_{a,n}=\sum_{T\in\mathfrak T_n}U_{a,T}\), where
\(\mathfrak T_n\) consists of rooted planar full binary trees
with \(n\) leaves and has cardinality \(\operatorname{Cat}_{n-1}\).
Each leaf is assigned \(F_a\), and
each internal vertex applies \(G_aQ_a\) to its two inputs.
This determines every coefficient through the explicit background
Green representation.

\subsection{Identification with the Calabi--Yau potential}
\label{subsec:CY-identification}
For a fixed \(a\) satisfying \eqref{eq:Catalan-smallness}, define
\(u_a(t)=\sum_{n\ge1}t^nU_{a,n}\) for \(0\le t\le1\).
The Catalan majorant gives uniform convergence in \(C^{2,\alpha}\)
on this interval, so \(u_a(t)\) is continuous in \(t\) and
\(u_a(0)=0\). The volume identity
\(c_a\int_X\mu=\int_XdV_a>0\) ensures \(c_a>0\).
Absolute convergence and bounded bilinearity give
\[
\begin{aligned}
u_a(t)&=-tG_af_a+G_aQ_a(u_a(t),u_a(t)),\\
\frac{(\omega_a+\ddc u_a(t))^2/2}{dV_a}
&=1+tf_a=(1-t)+t\frac{c_a\mu}{dV_a}>0.
\end{aligned}
\]
The {volume ratio on the left-hand side} is the determinant of the Hermitian endomorphism
\(\omega_a^{-1}(\omega_a+\ddc u_a(t))\). At each point of \(X\),
its eigenvalues vary continuously, start positive at \(t=0\), and
cannot pass through zero. {Taking \(t=1\) gives}
\begin{equation}\label{eq:positive-final}
\widetilde\omega_a:=\omega_a+\ddc u_a>0,
\qquad \frac{\widetilde\omega_a^2}{2}=c_a\mu.
\end{equation}
The form has \(C^{0,\alpha}\) coefficients and belongs to
\(\alpha_a\). For fixed \(a\), compactness gives a positive
lower bound relative to \(\omega_a\).

Let \(v_a\) be the smooth normalized potential given by
\eqref{eq:intro-Yau-target}, and put \(h=u_a-v_a\). Since the
two volume forms agree,
\[
\ddc h\wedge(\widetilde\omega_a+\omega_{\mathrm{CY},a})
=\widetilde\omega_a^2-\omega_{\mathrm{CY},a}^2=0.
\]
Integration by parts gives
\[
0=\int_Xh\,\ddc h\wedge
(\widetilde\omega_a+\omega_{\mathrm{CY},a})
=-\int_Xdh\wedge d^ch\wedge
(\widetilde\omega_a+\omega_{\mathrm{CY},a}).
\]
This identity is valid at the present regularity: approximate
\(u_a\) by smooth functions in \(C^2\), apply Stokes' theorem
to the resulting closed forms, and pass to the limit. With our
convention, \(dh\wedge d^ch=\ii\partial h\wedge\bar\partial h\)
is nonnegative. The two forms in the energy identity are positive
on the compact surface, by \eqref{eq:positive-final} and
\eqref{eq:intro-Yau-target}; hence \(dh=0\).
Both potentials have zero \(dV_a\)-mean, so
\begin{equation}\label{eq:CY-identification}
u_a=v_a,
\qquad \widetilde\omega_a=\omega_{\mathrm{CY},a}.
\end{equation}
{In particular, the sum is smooth and represents the unique
Ricci-flat metric in \(\alpha_a\).}

\begin{proof}[Completion of the proof of Theorem~\ref{thm:main}]
The background lies in \(\alpha_a\) by \eqref{eq:mod-class},
and Theorem~\ref{thm:explicit-global-green} specifies the normalized
Green operator in the recursion. The estimates
\eqref{eq:Fa-small}--\eqref{eq:Catalan-estimate} prove absolute
convergence and the coefficient bound. Equations
\eqref{eq:positive-final}--\eqref{eq:CY-identification} identify
the sum with \(v_a\). The metric representation and the bound
\(\|v_a\|_{\mathcal X_a}\le Ca^3\) now follow from
\eqref{eq:ua-majorant}. This proves all assertions.
\end{proof}

\begin{remark}[Scope of convergence]
\label{rem:fixed-a-convergence}
{Condition~\eqref{eq:Catalan-smallness} proves convergence near
the Kummer limit, but does not determine the maximal convergence
range. The Catalan tail estimate \eqref{eq:tree-tail} controls
the metric {truncation error} in \(C_{\mathrm{sc}}^{0,\alpha}\). Higher-order
estimates for the first remainder are proved separately in
Proposition~\ref{prop:higher-first-remainder}; the scope of these
estimates is discussed in Remark~\ref{rem:first-correction-scope}.}
\end{remark}

{
\subsection{The global first correction and its regular-locus limit}
\label{subsec:global-first-correction}
{The error after the first term of the recursion has a global estimate.}
{The first term's dependence} on \(a\) must be retained near the exceptional
curves; on fixed compact subsets of the regular orbifold locus,
the combined background and potential correction have a
parameter-independent leading coefficient.

\begin{theorem}[Global first correction]
\label{thm:global-first-correction}
Choose the threshold \(a_0\) in Theorem~\ref{thm:main} so that
\(\eta_a<1\) for \(0<a<a_0\), and put
\[
F_a=-G_af_a,\qquad
\omega_a^{[1]}=\omega_a^{\mathrm{mod}}+\ddc F_a,\qquad
r_a=u_a-F_a.
\]
Then \(\omega_a^{[1]}\) is a smooth K\"ahler form in
\(\alpha_a\), and on the entire resolved surface
\begin{equation}\label{eq:global-first-correction-remainder}
\begin{aligned}
\omega_{\mathrm{CY},a}
&=\omega_a^{[1]}+\ddc r_a,\\
\|r_a\|_{\mathcal X_a}
&\le4\mathfrak b_a\eta_a^2
\le4C_1^2C_2a^5.
\end{aligned}
\end{equation}
The first correction is nonzero, and its remainder obeys the
global norm comparison
\begin{equation}\label{eq:first-correction-relative-remainder}
\frac{\|r_a\|_{\mathcal X_a}}{\|F_a\|_{\mathcal X_a}}
\le4\mathfrak b_a\eta_a=O(a^2).
\end{equation}
In addition, with the gradient measured using
\(g_a^{\mathrm{mod}}\),
\begin{equation}\label{eq:first-correction-potential-remainder}
\|r_a\|_{C^0}+\|\nabla r_a\|_{C^0}
\le C\eta_a^2\le Ca^6.
\end{equation}
The exact normalized volume residual of the first approximation is
\begin{equation}\label{eq:first-correction-volume-residual}
\begin{aligned}
\frac{(\omega_a^{[1]})^2}{2c_a\mu}-1
&=\frac{Q_a(F_a,F_a)}{1+f_a},\\
\left\|\frac{(\omega_a^{[1]})^2}{2c_a\mu}-1\right\|_{\mathcal Y_a}
&\le\frac{1+\|\beta_a\|_{\mathcal Y_a}}{2c_a}\eta_a^2
\le Ca^6.
\end{aligned}
\end{equation}
All constants are independent of \(a\) in the stated range.
\end{theorem}

\begin{proof}
The fixed-point equation gives
\(r_a=G_aQ_a(u_a,u_a)\). Since
\(\|u_a\|_{\mathcal X_a}\le2\eta_a\), the definition of
\(\mathfrak b_a\) proves
\eqref{eq:global-first-correction-remainder}. The residual equals
\(-8\pi^2a^4\) on the flat exterior, so \(f_a\ne0\).
If \(\ddc F_a=0\), then \(A_aF_a=-f_a=0\), a contradiction.
Thus \(\eta_a>0\), and division gives
\eqref{eq:first-correction-relative-remainder}.
The source \(Q_a(u_a,u_a)\) has zero mean. Apply
\eqref{eq:Green-C0}, Lemma~\ref{lem:uniform-gradient}, and
\eqref{eq:Qa-uniform} to obtain
\eqref{eq:first-correction-potential-remainder}.
The relative Hermitian Hessian of \(F_a\) has operator norm
at most \(\eta_a<1\), so \(\omega_a^{[1]}>0\); the correction
is exact and preserves the K\"ahler class.
Finally, \(A_aF_a=-f_a\) and the quadratic volume identity yield
\((\omega_a^{[1]})^2/(2dV_a)=1+f_a+Q_a(F_a,F_a)\).
Divide by \(1+f_a\), and use
\((1+f_a)^{-1}=(1-\beta_a)/c_a\) and
\eqref{eq:Qa-uniform} for the claimed residual bound.
\end{proof}

\begin{proposition}[Higher-order control of the first remainder]
\label{prop:higher-first-remainder}
Choose the threshold \(a_0\) so that \(2\eta_a\le1/2\).
For every fixed integer \(k\ge0\), there is a constant \(C_k\),
independent of \(0<a<a_0\), such that
\begin{equation}\label{eq:higher-first-global}
\begin{aligned}
\|H_{a,u_a}\|_{C_{\mathrm{sc}}^{k,\alpha}}
+\|H_{a,F_a}\|_{C_{\mathrm{sc}}^{k,\alpha}}
&\le C_ka^3,\\
\|H_{a,r_a}\|_{C_{\mathrm{sc}}^{k,\alpha}}
&\le C_ka^5.
\end{aligned}
\end{equation}
For every fixed \(K\Subset Y_{\mathrm{reg}}\), identified with
its lift to \(X\), there is \(C_{K,k}\) such that
\begin{equation}\label{eq:higher-first-local}
\|u_a\|_{C^{k+2,\alpha}(K)}\le C_{K,k}a^4,
\qquad
\|r_a\|_{C^{k+2,\alpha}(K)}\le C_{K,k}a^6.
\end{equation}
In particular, the metric error of the first approximation on
\(K\) is \(O_{C^{k,\alpha}(K)}(a^6)\).
\end{proposition}

\begin{proof}
The fixed-point equation, the residual estimate, and
\eqref{eq:Green-C0} and \eqref{eq:gradient-bound} give
\[
\|u_a\|_{C^0}+\|\nabla u_a\|_{C^0}
\le C\bigl(\|f_a\|_{\mathcal Y_a}
+\|Q_a(u_a,u_a)\|_{\mathcal Y_a}\bigr)\le Ca^4.
\]
Use enlarged holomorphic adapted charts of physical size
\(s\asymp s_a(x_0)\). On the core these are fixed smooth
resolution charts of the unit Eguchi--Hanson space after dilation;
for example the normal fiber coordinate changes by
\(t=a^2T\), while the exceptional coordinate is unchanged.
Dyadic and fixed outer charts can likewise be chosen holomorphic.
Lemma~\ref{lem:uniform-adapted-charts-final} gives uniform
ellipticity and every fixed-order bound for the scaled background.
On such a chart set
\(U=s^{-2}(u_a-u_a(x_0))\), with coordinate pullback understood.
Then \(\|U\|_{C^0}\le Ca^4/s\).

Write \(B\) for the scaled background Hermitian matrix and
\(E=\partial\bar\partial U\). The volume equation is
\(\det(B+E)=(1+f_a)\det B\).
The initial relative Hessian bound and the choice of \(a_0\)
make the averaged cofactor matrix uniformly elliptic with
uniform \(C^{0,\alpha}\) coefficients. The equation
\[
\left(\int_0^1\operatorname{cof}(B+tE)\,dt\right):E
=f_a\det B
\]
and the interior Schauder estimate yield
\(\|U\|_{C^{2,\alpha}}\le C(a^4/s+a^4)\le Ca^4/s\)
on a smaller chart. {This also controls the full real Hessian.} To retain this small bound at higher
orders, differentiate the exact determinant equation. For a real
coordinate derivative \(\partial_\ell\), the equation is
\[
\begin{aligned}
\operatorname{cof}(B+E):
\partial\bar\partial(\partial_\ell U)
={}&(\partial_\ell f_a)\det B
+f_a\partial_\ell\det B\\
&+[\operatorname{cof}(B)-\operatorname{cof}(B+E)]
:\partial_\ell B.
\end{aligned}
\]
Its right-hand side is \(O(a^4/s)\) in \(C^{0,\alpha}\).
Repeated differentiation and interior Schauder estimates on
nested charts therefore give
\(\|U\|_{C^{k+2,\alpha}}\le C_ka^4/s\).
The same bounds for the scaled linear equation of \(F_a\)
follow from its gradient estimate and the smooth residual bounds.
Conversion to covariant derivatives is uniform on these charts.
Since \(s\ge ca\), taking the adapted covering proves the
first line of \eqref{eq:higher-first-global}.

The quadratic source consequently satisfies
\(\|Q_a(u_a,u_a)\|_{C_{\mathrm{sc}}^{k,\alpha}}\le C_ka^6\).
By \eqref{eq:first-correction-potential-remainder}, the scaled
function \(s^{-2}(r_a-r_a(x_0))\) has supremum at most
\(Ca^6/s\). Its background linear equation has right-hand side
\(Q_a(u_a,u_a)\), with the preceding scaled derivative bounds.
Higher-order linear Schauder estimates give
\(\|H_{a,r_a}\|_{C^{k,\alpha}}\le C_k(a^6/s+a^6)\)
in the scaled charts, hence the second line of
\eqref{eq:higher-first-global}. Enlarged charts before restriction
also control all close pairs in the intrinsic H\"older seminorm.

For the local conclusion, take a slightly larger fixed regular
compact neighborhood of \(K\). There \(s_a\) is bounded below,
so the estimates for \(u_a\), together with its \(C^0\) bound,
give \(\|u_a\|_{C^{k+2,\alpha}}\le C_{K,k}a^4\).
The local quadratic source is therefore \(O_{C^{k,\alpha}}(a^8)\).
Apply fixed-scale interior Schauder estimates to
\(A_ar_a=Q_a(u_a,u_a)\), using the global
\(\|r_a\|_{C^0}\le Ca^6\), to obtain
\eqref{eq:higher-first-local}.
\end{proof}

\begin{theorem}[The leading coefficient on the regular locus]
\label{thm:regular-first-coefficient}
Let \(Y_{\mathrm{reg}}=Y\setminus\{p_1,\ldots,p_{16}\}\), and
{define a function on this locus by}
\begin{equation}\label{eq:regular-leading-potential}
H(x)=-\frac{\pi^2}{2}\sum_{\nu=1}^{16}G_Y(x,p_\nu).
\end{equation}
For every fixed integer \(k\ge0\) and
\(K\Subset Y_{\mathrm{reg}}\), identify \(K\)
with its lift to \(X\) by the resolution map. Then
\begin{equation}\label{eq:regular-first-coefficient-limit}
\omega_{\mathrm{CY},a}
=\omega_0+a^4\ddc H+o_{C^{k,\alpha}(K)}(a^4).
\end{equation}
The coefficient \(\ddc H\) is independent of the radial cutoff.
It is nonzero and satisfies
\(\Lambda_{\omega_0}\ddc H=-8\pi^2\) on
\(Y_{\mathrm{reg}}\).
\end{theorem}

\begin{proof}
Write \(\beta_a=a^4\beta^{(4)}\), with
\(\beta^{(4)}=\psi(\sqrt\rho)/(4I_\psi)\) on each annulus,
extended by zero elsewhere. Then
\(a^{-4}f_a\to\beta^{(4)}-8\pi^2\) smoothly on fixed
regular compact sets. By \eqref{eq:Green-C0},
\(V_a=a^{-4}F_a\) is uniformly bounded in \(C^0\).
The backgrounds converge smoothly to the flat metric away from
the exceptional curves. Interior elliptic estimates for
\(A_aV_a=-a^{-4}f_a\) give subsequential local smooth
convergence to a bounded solution {of}
\[
A_0F_0=8\pi^2-\beta^{(4)},\qquad A_0=-\tfrac12\Delta_0.
\]
Near every orbifold point the right-hand side is constant.
Subtracting a smooth quadratic particular solution {from the}
{lifted function leaves} a bounded harmonic function on a punctured
four-dimensional ball, whose singularity is removable.
Thus \(F_0\) extends smoothly as an orbifold function.

The normalization also passes to the limit. Let \(N_\delta\)
be the union of the sixteen orbifold balls of radius \(\delta\).
Since \(dV_a=(1-\beta_a)\mu\le\mu\), their resolved preimages
have volume at most \(C\delta^4\), uniformly in \(a\).
The integral of the bounded function \(V_a\) over these
preimages is therefore at most \(C\delta^4\) in absolute value.
On the complement the functions and measures converge.
First let \(a\to0\), and then \(\delta\to0\), to obtain
\(\int_YF_0\,dV_0=0\).
Two such limits differ by a globally smooth harmonic orbifold
function of zero mean, hence {coincide}. This proves convergence of
the whole family. Schauder estimates for the difference equation
give \(V_a\to F_0\) in every fixed \(C^{k+2,\alpha}\)-norm on
every fixed regular compact set.

{To recover the metric coefficient, the} background contribution must {also} be included. Put
\(m^{(4)}(\rho)=\int_0^\rho s\beta^{(4)}(s)\,ds\), and
define on each punctured neighborhood
\[
h_\psi(\rho)
=\int_{r_4^2}^{\rho}
\frac{1-2m^{(4)}(s)}{4s^2}\,ds,
\]
extending it by zero on the exterior. Expansion of the exact
radial coefficients gives, on fixed regular compact sets,
\[
\omega_a^{\mathrm{mod}}
=\omega_0+a^4\ddc h_\psi+O_{C^{k,\alpha}}(a^8).
\]
Here \(A_0h_\psi=\beta^{(4)}\) away from the orbifold points,
and \(h_\psi=-1/(4r_\nu^2)+\text{constant}\) near \(p_\nu\).
Consequently \(h_\psi+F_0\) satisfies
\(A_0(h_\psi+F_0)=8\pi^2\) and has singular part
\(-1/(4r_\nu^2)\) at each point.
By \eqref{app:eq:GY-equation}, the function \(H\) in
\eqref{eq:regular-leading-potential} has exactly the same
equation and singular parts. Their difference has removable
harmonic singularities and is constant on the compact orbifold.
Thus \(\ddc(h_\psi+F_0)=\ddc H\).
Finally, Proposition~\ref{prop:higher-first-remainder} gives
\(\ddc(u_a-F_a)=O_{C^{k,\alpha}(K)}(a^6)\), proving
\eqref{eq:regular-first-coefficient-limit}. The trace assertion
follows from \(A_0H=8\pi^2\).
\end{proof}

\begin{remark}[Two meanings of the first term]
\label{rem:first-correction-scope}
Theorem~\ref{thm:global-first-correction} is uniform on the entire
smooth surface, including the exceptional curves, and retains
the \(a\)-dependent first term \(F_a\).
{Theorem~\ref{thm:regular-first-coefficient} gives a
parameter-independent \(a^4\) coefficient on fixed regular compact sets.}
The singular function \(H\) is not a global smooth potential on
\(X\). {The expansion by \(a^4\ddc H\) therefore applies only on
regular compact sets.} The norm ratio in
\eqref{eq:first-correction-relative-remainder} is a global norm
comparison, not a pointwise relative error.
Proposition~\ref{prop:higher-first-remainder} controls the actual
solution and its first remainder in every fixed derivative order.
It does not establish termwise absolute convergence of the whole
series in all higher-order norms or higher-order tail constants
for arbitrary truncation orders. The local \(O(a^6)\) remainder
is relative to the \(a\)-dependent first approximation; the
convergence of \(a^{-4}F_a\) to \(F_0\) has no explicit rate here,
so the expansion with fixed coefficient \(H\) retains a
little-oh remainder.
\end{remark}
}

\section{Evaluation and truncation}
\label{sec:computational-realization}

{We record truncation formulas and the additional bounds needed
for certified numerical evaluation; {these additional bounds} are not needed for
Theorem~\ref{thm:main}. Write \(\omega_a=\omega_a^{\mathrm{mod}}\).
Optional symmetry reductions, angular bases, and Fourier moments
are collected in Appendix~\ref{app:angular-evaluation}.}

\subsection{The background and its residual}
\label{subsec:computing-background}
The background is evaluated from the fixed cutoff profile and the
radial formulas \eqref{eq:radial-density}--\eqref{eq:radial-coefficients}. They require
one-dimensional quadrature and elementary operations. Together with
the residual formula \eqref{eq:fa-explicit}, they give
\[
\begin{aligned}
\beta_a&=a^4\beta^{(4)},
\qquad \beta^{(4)}(\rho)=\frac{\psi(\sqrt\rho)}{4I_\psi},\\
a^{-4}f_a&\longrightarrow f^{(4)}:=\beta^{(4)}-8\pi^2.
\end{aligned}
\]
{The convergence holds in every fixed H\"older norm on the transition
region.}

\subsection{Evaluation of the periodic kernel}
\label{subsec:evaluate-Ewald}

The periodic Green kernel is evaluated through
\eqref{app:eq:Ewald}. For each Ewald parameter \(\sigma>0\),
the real-space and Fourier sums and their differentiated sums have
Gaussian tails. Truncating at \(|m|\le M\) and \(|n|\le N\) gives
\[
\left|
\partial^\beta G_T
-
\partial^\beta G_T^{M,N}
\right|
\le
C_\beta
\left[
(1+M)^{C_\beta}
e^{-cM^2/\sigma}
+
(1+N)^{C_\beta}
e^{-c\sigma N^2}
\right].
\]
Here the constants are for a fixed Ewald parameter and a fixed
compact set away from the lattice singularities. The same-sphere
singularity must first be separated analytically as in
Appendix~\ref{app:ewald}. Evaluation of boundary matrix elements
also requires quadrature and angular truncation estimates; the Ewald
tail alone does not certify the complete boundary computation.

\subsection{Evaluation of the cap kernels}
\label{subsec:evaluate-cap-kernels}

{For every pair \((p,q)\), the cap calculation begins with the regular
radial solution \(w_{p,q,a}^{\mathrm{reg}}\).} Near the exceptional curve it is initialized by its convergent
Frobenius expansion, evaluated at \(0<\rho_*<\min\{a^2,r_1^2\}\), {as given in}
Lemma~\ref{app:lem:regular-Frobenius}.  It is then continued across
the radial transition region by the one-dimensional equation
\eqref{app:eq:homogeneous-radial}.

The reduction-of-order formula \eqref{app:eq:wD} gives the
zero-boundary second solution by one-dimensional quadrature.
The product \eqref{app:eq:radial-Green} then gives the radial
Green kernel, and \eqref{app:eq:cap-DtN} gives the cap
Dirichlet-to-Neumann coefficient from the boundary value and
derivative of the regular solution.
{The continuation is specified by the convergent iterated-integral
series \eqref{app:eq:Volterra-series}; a numerical ODE solver is one
way to evaluate {the continuation}.}

\subsection{Evaluation of the interface inverse}
\label{subsec:evaluate-interface}

Choose an even spherical cutoff \(L\) for which the high-mode
contraction in
\eqref{eq:high-mode-uniform-contraction}
holds.

The finite low-mode matrix \(M_{a,L}\) is assembled from Ewald boundary contractions, cap
Dirichlet-to-Neumann coefficients, and high-mode Schur corrections.

{Once \(L\) is fixed, the matrix dimension is independent of the
desired numerical accuracy. Formula~\eqref{eq:M-inverse} gives
its inverse in terms of its exact entries; numerical evaluation
may use a stable factorization.}
For the high modes, put
\(K_{a,L}=(2\mathcal A_\Sigma)^{-1}Q_L\mathcal K_aQ_L\).
Then \(\|K_{a,L}\|\le q<1\), and \eqref{eq:NHH-inverse}
gives a convergent Neumann series.
Retaining the terms indexed by \(0,\ldots,J\) gives
\[
\left\|
\mathcal N_{HH}^{-1}
-
\sum_{j=0}^{J}
(-K_{a,L})^j
(2\mathcal A_\Sigma)^{-1}Q_L
\right\|
\le
\frac{
q^{J+1}
}{
1-q
}
\,
\|(2\mathcal A_\Sigma)^{-1}Q_L\|.
\]
Hence the interface matching may be evaluated with a geometric tail
bound.

\subsection{The first nonlinear correction}
\label{subsec:first-nonlinear-correction}
{The first nonlinear term after the global first
approximation of Theorem~\ref{thm:global-first-correction} is}
\(U_{a,2}=G_aQ_a(U_{a,1},U_{a,1})\), where
\(U_{a,1}=-G_af_a\). The estimates give
\(\|U_{a,1}\|_{\mathcal X_a}=O(a^3)\) and
\(\|U_{a,2}\|_{\mathcal X_a}=O(a^5)\).
The general coefficient bound is \eqref{eq:Catalan-estimate}.
{The general coefficient bound suffices} for convergence and need not be sharp;
{it} does not imply
\(U_{a,n}=O(a^{4n})\) globally.

\subsection{Truncation of the binary-tree series}
\label{subsec:tree-truncation}

Put \(t_a=\eta_a\mathfrak b_a\) and
\(u_a^{[M]}=\sum_{n=1}^{M}U_{a,n}\).
For an admissible \(a\), \(4t_a<1\).
Using \(\operatorname{Cat}_m\le4^m\) and summing the geometric
majorant gives
\begin{equation}\label{eq:tree-tail}
\|u_a-u_a^{[M]}\|_{\mathcal X_a}
\le
\eta_a
\frac{
(4t_a)^M
}{
1-4t_a
}.
\end{equation}
Once upper bounds for \(\eta_a\) and \(\mathfrak b_a\) are
available, this gives a stopping criterion for the exact nonlinear
series.

\subsection{Scope of the numerical error estimates}
\label{subsec:deterministic-evaluation}
For fixed admissible \(a\), the Catalan bound controls {the error from truncating the exact}
tree {series} in \(\mathcal X_a\), hence the metric error in
\(C_{\mathrm{sc}}^{0,\alpha}\). The Ewald and Volterra formulas and
the high-mode Neumann series give separate convergent representations
of their respective components. To turn these formulas into a
certified computation at a specified numerical value of \(a\), one
must bound the constants and propagate the errors of the components.
In particular, a finite Neumann truncation still acts on infinitely
many angular modes and must be accompanied by an angular cutoff.

\begin{proposition}[Conditional error certificate]
\label{prop:computability}
Suppose \(\eta_a\) and \(\mathfrak b_a\) have certified upper bounds
\(\overline\eta_a\) and \(\overline b_a\), with
\(4\overline\eta_a\overline b_a<1\). If computed smooth, mean-zero
terms \(\widetilde U_{a,n}\), \(1\le n\le M\), satisfy
\(\|\widetilde U_{a,n}-U_{a,n}\|_{\mathcal X_a}\le\epsilon_n\), then
\[
\left\|\omega_a^{-1}
\left(\omega_a+\ddc\sum_{n=1}^M\widetilde U_{a,n}
-\omega_{\mathrm{CY},a}\right)\right\|_{C_{\mathrm{sc}}^{0,\alpha}}
\le\sum_{n=1}^M\epsilon_n+
\overline\eta_a\frac{(4\overline\eta_a\overline b_a)^M}
{1-4\overline\eta_a\overline b_a}.
\]
\end{proposition}
\begin{proof}
Apply the triangle inequality and \eqref{eq:tree-tail}. The positive
scalar majorant is increasing in both of its arguments.
\end{proof}

Certification of the \(\epsilon_n\) requires quadrature and angular
tail bounds, radial continuation errors, bounds for the inverse
matrices, and error propagation through the finite tree recursion.
For example, if a matrix approximation satisfies
\(\|\widetilde M^{-1}\|\le K\) and
\(\|M-\widetilde M\|\le\epsilon<K^{-1}\), then
\(\|M^{-1}-\widetilde M^{-1}\|
\le K^2\epsilon/(1-K\epsilon)\).
The preceding proofs establish finite analytic constants and a positive
convergence threshold. Numerical upper bounds for these constants, a
certified numerical value of \(a_0\), and certified component errors
remain to be determined; an interval-arithmetic implementation is not
supplied here. Likewise, smoothness of the exact
solution does not by itself give a \(C^k\) error bound for arbitrary
finite approximations: such a bound requires a differentiated
convergence argument or sufficiently strong bounds for their residuals.
The proved truncation certificate above is in the displayed
\(C_{\mathrm{sc}}^{0,\alpha}\) metric norm.
{The higher-order estimates in
Proposition~\ref{prop:higher-first-remainder} concern the first
approximation \(\omega_a^{[1]}\); they do not provide
higher-order certificates for arbitrary computed truncations.}

\section{The elliptic K3 setting}
\label{sec:elliptic-outlook}
{We have also completed the Green-operator construction
and the convergent nonlinear recursion on collapsing elliptic K3
surfaces in the setting below. We state this extension here and
outline its analytic mechanism using
Proposition~\ref{prop:abstract-recursion}. The local Green kernels,
their global assembly, and the proofs of the uniform estimates are
not expanded in this section.}
Gross--Wilson~\cite{GrossWilson} construct approximate Ricci-flat
metrics by gluing two complementary local models. On the smooth
locus, the \emph{semi-flat metric}~\cite{GreeneShapereVafaYau1990}
is an explicit Ricci-flat K\"ahler metric with flat,
translation-invariant elliptic fibers; the periods determine its
variation over the base~\cite[Example~2.2]{GrossWilson}.
It models the collapsing regular fibers but does not extend
smoothly across the nodal fibers.

Near each \(I_1\) fiber, the \emph{Ooguri--Vafa (OV)
metric}~\cite{OoguriVafa1996} provides an explicit hyper-K\"ahler
model, built from a periodic Gibbons--Hawking potential, that is
smooth on the total space across the nodal
fiber~\cite[Section~3]{GrossWilson}. The nonzero Fourier modes
of this potential decay exponentially on fixed matching annuli
as the fiber area tends to zero, while its zero mode produces the corresponding
semi-flat model~\cite[Lemma~3.1]{GrossWilson}.
Matching the periods and translation data and gluing these OV
models to the semi-flat metric yields a smooth
global K\"ahler background with exponentially small
Monge--Amp\`ere residual~\cite[Theorem~4.5]{GrossWilson}.
This exponential accuracy, together with the Green-operator
bounds assumed below, is the relevant feature for the recursion.

For this discussion, let \(\pi:X\to\mathbb P^1\) be an elliptic
K3 surface with zero
section \(S\) and twenty-four singular fibers of type \(I_1\).
Write \([F]\) for the fiber class and consider
\(\alpha_\ee=\ee[S]+(\ee^{-1}+\ee)[F]\), so that
\(\alpha_\ee^2/2=1\). Normalize
\(\mu=\frac14\Omega\wedge\bar\Omega\) by \(\int_X\mu=1\).
For a smooth K\"ahler background
\(\widehat\omega_\ee\in\alpha_\ee\), put
\(dV_\ee=\widehat\omega_\ee^2/2\) and define
\[
\begin{aligned}
A_\ee&=-\Lambda_{\widehat\omega_\ee}\ddc,
\qquad f_\ee=\frac{\mu}{dV_\ee}-1,\\
Q_\ee(u,v)&=\frac{\ddc u\wedge\ddc v}{\widehat\omega_\ee^2}.
\end{aligned}
\]
Both \(f_\ee\) and \(Q_\ee(u,v)\) have zero \(dV_\ee\)-mean.
Let \(\mathbf G_\ee\) denote the inverse of \(A_\ee\) on
functions of zero \(dV_\ee\)-mean, with its output normalized to
have zero mean as well.
The exact equation is then
\(u=-\mathbf G_\ee f_\ee+\mathbf G_\ee Q_\ee(u,u)\).

Fix a smooth reference metric \(g_*\), and let \(\mathcal E_\ee\)
be the {space of real-valued} \(H^5(X,g_*)\) functions of zero \(dV_\ee\)-mean,
with the \(H^5(X,g_*)\) norm. Suppose that, for constants
\(C,c>0\) and \(M\ge0\), the following bounds hold for every
sufficiently small \(\ee>0\):
\begin{align}
\|\mathbf G_\ee f_\ee\|_{\mathcal E_\ee}
&\le Ce^{-c/\ee},
&\|\mathbf G_\ee Q_\ee\|_{\mathrm{bil}}
&\le C\ee^{-M},\label{eq:elliptic-analytic-inputs}\\
\|\widehat\omega_\ee^{-1}\ddc v\|_{C^0}
&\le C\ee^{-M}\|v\|_{\mathcal E_\ee}
&&\text{for }v\in\mathcal E_\ee.\label{eq:elliptic-Hessian-input}
\end{align}
The bilinear norm in \eqref{eq:elliptic-analytic-inputs} is on
\(\mathcal E_\ee\times\mathcal E_\ee\to\mathcal E_\ee\).
The endomorphism norm in \eqref{eq:elliptic-Hessian-input} is
measured with respect to \(\widehat\omega_\ee\).
An exponentially small residual in \(H^3(X,g_*)\), together with a
polynomial bound for \(\mathbf G_\ee:H^3_0\to H^5_0\), yields the first estimate,
after decreasing the exponential rate; the subscript denotes zero
\(dV_\ee\)-mean.

These assumptions give
\(4\|\mathbf G_\ee f_\ee\|_{\mathcal E_\ee}
\|\mathbf G_\ee Q_\ee\|_{\mathrm{bil}}\to0\).
Consequently, the {series defined by the} recursion
\[
U_{\ee,1}=-\mathbf G_\ee f_\ee,\qquad
U_{\ee,n}=\mathbf G_\ee\sum_{j=1}^{n-1}
Q_\ee(U_{\ee,j},U_{\ee,n-j}),\quad n\ge2,
\]
converges {absolutely} in \(\mathcal E_\ee\), with sum \(u_\ee\) of norm
\(O(e^{-c/\ee})\){. The sum satisfies}
\((\widehat\omega_\ee+\ddc u_\ee)^2/2=\mu\).
The homotopy \(\sum_{n\ge1}t^nU_{\ee,n}\) has volume ratio
\((1-t)+t\mu/dV_\ee>0\) for \(0\le t\le1\); it therefore
preserves positivity, as in Section~\ref{subsec:CY-identification}.
The Hessian estimate controls the relative metric correction.
Since \(H^5\) embeds into \(C^{2,\alpha}\) for \(0<\alpha<1\)
in real dimension four, comparison with the smooth Calabi--Yau potential, as in
Section~\ref{subsec:CY-identification}, identifies this sum with
the normalized potential of the Ricci-flat metric in \(\alpha_\ee\).
{In the completed elliptic construction, the local Green
kernels and their global assembly verify
\eqref{eq:elliptic-analytic-inputs}--\eqref{eq:elliptic-Hessian-input}.
Thus the recursion above gives a convergent representation of the
Ricci-flat metric for all sufficiently small \(\ee>0\).
We record this conclusion without the detailed kernel formulas and
proofs of the estimates, which are omitted from the present paper.}

\clearpage
\appendix
\section{The periodic Green kernel and Ewald representation}
\label{app:ewald}

In this appendix we collect the explicit formulas for the periodic
Green kernel used in the exterior problem. We derive an absolutely
convergent Ewald representation, justify termwise differentiation,
normalize the periodic multipoles at the Kummer points, and record
the Fourier-moment formula used in
Section~\ref{sec:computational-realization}.

Throughout this appendix, \(T=\mathbb R^4/\mathbb Z^4\) has volume
one and \(\Delta_0=\sum_{j=1}^4\partial_{x_j}^2\). Thus
\(-\Delta_0\) is nonnegative, and its fundamental solution on
\(\mathbb R^4\) satisfies
\(-\Delta_0(1/(4\pi^2|x|^2))=\delta_0\).

\subsection{The normalized torus Green function}

Let \(G_T\) be the unique distribution satisfying
\(-\Delta_0G_T=\delta_0-1\) and \(\int_TG_T\,dx=0\).
Its Fourier expansion is
\begin{equation}\label{app:eq:GT-Fourier}
G_T(x)=\sum_{n\in\mathbb Z^4\setminus\{0\}}
\frac{e^{2\pi\ii n\cdot x}}{4\pi^2|n|^2},
\end{equation}
where the series is understood distributionally. For pointwise
evaluation and differentiation we use the heat kernel. Poisson
summation gives its two representations:
\begin{equation}\label{app:eq:Poisson}
K_t(x)=\sum_{n\in\mathbb Z^4}e^{-4\pi^2|n|^2t}
e^{2\pi\ii n\cdot x}
=\frac1{(4\pi t)^2}\sum_{m\in\mathbb Z^4}
\exp\!\left(-\frac{|x+m|^2}{4t}\right).
\end{equation}
Since the zero Fourier mode of \(K_t\) is one,
\begin{equation}\label{app:eq:heat-Green}
G_T(x)=\int_0^\infty\bigl(K_t(x)-1\bigr)\,dt,
\qquad x\notin\mathbb Z^4.
\end{equation}
Fix an arbitrary parameter \(\sigma>0\).

\begin{proposition}[Ewald representation]
\label{app:prop:Ewald}
For every \(x\in\mathbb R^4\setminus\mathbb Z^4\), one has
\begin{equation}\label{app:eq:Ewald}
\begin{aligned}
G_T(x)={}&\sum_{m\in\mathbb Z^4}
\frac{\exp\!\left(-|x+m|^2/(4\sigma)\right)}
{4\pi^2|x+m|^2}-\sigma\\
&+\sum_{n\in\mathbb Z^4\setminus\{0\}}
\frac{e^{-4\pi^2\sigma|n|^2}}{4\pi^2|n|^2}
e^{2\pi\ii n\cdot x}.
\end{aligned}
\end{equation}
Both series are absolutely and locally uniformly convergent away
from \(\mathbb Z^4\).
\end{proposition}

\begin{proof}
Split \eqref{app:eq:heat-Green} at \(t=\sigma\). For \(r>0\),
direct integration gives
\(\int_0^\sigma e^{-r^2/(4t)}(4\pi t)^{-2}\,dt
=e^{-r^2/(4\sigma)}/(4\pi^2r^2)\).
Using the two representations in \eqref{app:eq:Poisson}, we obtain
\[
\begin{aligned}
\int_0^\sigma(K_t(x)-1)\,dt
&=\sum_{m\in\mathbb Z^4}\int_0^\sigma
\frac{e^{-|x+m|^2/(4t)}}{(4\pi t)^2}\,dt-\sigma\\
&=\sum_{m\in\mathbb Z^4}
\frac{e^{-|x+m|^2/(4\sigma)}}{4\pi^2|x+m|^2}-\sigma,\\
\int_\sigma^\infty(K_t(x)-1)\,dt
&=\sum_{n\ne0}\frac{e^{-4\pi^2\sigma|n|^2}}
{4\pi^2|n|^2}e^{2\pi\ii n\cdot x}.
\end{aligned}
\]
Adding the two expressions proves the formula.
\end{proof}

The right-hand side of \eqref{app:eq:Ewald} is independent of
\(\sigma\). Thus the Ewald parameter may be chosen solely for
computational convenience.

\subsection{Differentiation and truncation}

\begin{proposition}[Termwise differentiation]
\label{app:prop:Ewald-derivatives}
Let \(K\Subset\mathbb R^4\setminus\mathbb Z^4\) and let
\(\beta\) be a multi-index. Then \(\partial^\beta G_T\) is obtained
by differentiating both series in \eqref{app:eq:Ewald} term by term.
The differentiated series converge absolutely and uniformly on \(K\).
\end{proposition}

\begin{proof}
For the reciprocal-lattice series, differentiation introduces only a
polynomial factor in \(|n|\), while \(e^{-4\pi^2\sigma|n|^2}\)
provides Gaussian decay. For the real-space series, every derivative
of \(e^{-|x+m|^2/(4\sigma)}/|x+m|^2\) is a finite sum of polynomial
and rational factors multiplied by a Gaussian. Since \(K\) stays a
positive distance from the lattice, the terms are dominated by
\(C_{\beta,\sigma}(1+|m|)^{N_\beta}e^{-c_\sigma|m|^2}\).
The result follows from the Weierstrass test.
\end{proof}

In particular, if the two Ewald sums are truncated at \(|m|\le M\)
and \(|n|\le N\), then on every compact subset away from the lattice,
\[
\left|\partial^\beta G_T-\partial^\beta G_T^{M,N}\right|
\le C_\beta\left[(1+M)^{N_\beta}e^{-cM^2/\sigma}
+(1+N)^{N_\beta}e^{-c\sigma N^2}\right].
\]
Thus every derivative used in the construction has an exponentially
convergent representation.

\subsection{The regular part at the origin}

Set \(\mathscr H(x)=G_T(x)-1/(4\pi^2|x|^2)\). This function
extends real analytically across the origin. Indeed, the \(m=0\)
term in the real-space Ewald sum satisfies
\[
\frac{e^{-|x|^2/(4\sigma)}-1}{4\pi^2|x|^2}
=\frac1{4\pi^2}\sum_{j=1}^\infty
\frac{(-1)^j|x|^{2j-2}}{j!(4\sigma)^j},
\]
while all remaining terms are analytic near \(x=0\). Its value at
the origin is therefore
\[
\mathscr H(0)=\sum_{m\ne0}\frac{e^{-|m|^2/(4\sigma)}}
{4\pi^2|m|^2}
+\sum_{n\ne0}\frac{e^{-4\pi^2\sigma|n|^2}}{4\pi^2|n|^2}
-\sigma-\frac1{16\pi^2\sigma}.
\]
Away from the origin, \(-\Delta_0G_T=-1\). Since the Euclidean
fundamental solution is harmonic there, \(\Delta_0\mathscr H=1\).
As \(\Delta_0|x|^2=8\), we may write
\[
\mathscr H(x)=h_0+\frac{|x|^2}{8}+\mathscr H_{\mathrm{har}}(x),
\qquad
\mathscr H_{\mathrm{har}}=\sum_{\substack{\ell\ge4\\\ell\ {\rm even}}}
H_\ell,\quad H_\ell\in\mathcal H_\ell^{B_4}.
\]
Here \(h_0=\mathscr H(0)\), and \(\mathscr H_{\mathrm{har}}\) is
harmonic near the origin and vanishes to order at least four. The
expansion follows from invariance of \(\mathscr H\) under the
hyperoctahedral group \(B_4=(\mathbb Z_2)^4\rtimes S_4\) of the
square lattice.

\subsection{The quotient Green kernel}

Let \(q:T\to Y=T/\{\pm1\}\) be the quotient map. For a general
source the kernel is
\(G_Y(x,y)=G_T(\widetilde x-\widetilde y)
+G_T(\widetilde x+\widetilde y)\), as in \eqref{eq:GY-general}.
Source derivatives at a fixed point are taken in these covering
coordinates. For an orbifold point \(p_\nu\), choose a two-torsion
lift \(\widetilde p_\nu\in T\) and define
\(G_Y(x,p_\nu)=2G_T(\widetilde x-\widetilde p_\nu)\).
Because \(G_T\) is even and \(\widetilde p_\nu\) is two-torsion,
this is independent of the choice of lift. The factor two is forced
by \(\operatorname{Vol}(Y)=1/2\), and one has
\begin{equation}\label{app:eq:GY-equation}
-\Delta_0G_Y(\,\cdot\,,p_\nu)=\delta_{p_\nu}-2,
\qquad \int_YG_Y(x,p_\nu)\,dV_0(x)=0.
\end{equation}
Near \(p_\nu\),
\(G_Y(x,p_\nu)=1/(2\pi^2r_\nu^2)+2\mathscr H(x-p_\nu)\).
The coefficient \(1/(2\pi^2)\) is twice the Euclidean coefficient
because the link of an \(A_1\)-orbifold point is \(S^3/\{\pm1\}\).

\subsection{Periodic multipoles}

Let \(P_{\ell,\gamma}(x)=r^\ell Y_{\ell,\gamma}(\theta)\) be a
homogeneous harmonic polynomial of even degree \(\ell\).
Hobson's formula and the sign from source differentiation give
\[
\begin{aligned}
P_{\ell,\gamma}(\partial_x)|x|^{-2}
&=(-1)^\ell2^\ell\ell!\,r^{-\ell-2}Y_{\ell,\gamma}(\theta),\\
P_{\ell,\gamma}(\partial_y)G_Y(x,y)\big|_{y=p_\nu}
&=\frac{2^\ell\ell!}{2\pi^2}
r_\nu^{-\ell-2}Y_{\ell,\gamma}+O(1).
\end{aligned}
\]
We therefore define the normalized periodic multipole by
\begin{equation}\label{app:eq:multipole}
G_{\nu,\ell,\gamma}(x)=\frac{2\pi^2}{2^\ell\ell!}
P_{\ell,\gamma}(\partial_y)G_Y(x,y)\big|_{y=p_\nu}.
\end{equation}
It satisfies
\(G_{\nu,\ell,\gamma}=r_\nu^{-\ell-2}Y_{\ell,\gamma}+O(1)\)
near \(p_\nu\). For the positive even degrees used here,
differentiation annihilates the constant term in
\eqref{app:eq:GY-equation}, so
\(\Delta_0G_{\nu,\ell,\gamma}=0\) away from \(p_\nu\).

\subsection{The symmetric sixteen-point sum}

Let \(\mathcal P=(\frac12\mathbb Z/\mathbb Z)^4\) and define the
completely symmetric monopole sum
\(\mathcal G_{\mathrm K}(x)=\sum_{p\in\mathcal P}G_Y(x,p)\).
Since \(\sum_{p\in\mathcal P}e^{-2\pi\ii n\cdot p}\) equals
\(16\) for \(n\in(2\mathbb Z)^4\) and zero otherwise,
one has, in the distributional sense,
\[
\mathcal G_{\mathrm K}(x)=\frac2{\pi^2}
\sum_{k\in\mathbb Z^4\setminus\{0\}}
\frac{e^{4\pi\ii k\cdot x}}{|k|^2}.
\]
This explains the \(4\pi k\) Fourier frequencies appearing in
Section~\ref{sec:computational-realization}.

\subsection{Fourier moments on a Kummer cell}

Let \(C=[-1/4,1/4]^4\), \(D_R=C\setminus B_R\), and
\(0<R<1/4\). For a homogeneous harmonic polynomial \(P_\ell\)
of degree \(\ell\), define
\(\Lambda_{P_\ell}(n)=\int_{D_R}P_\ell(x)e^{-4\pi\ii n\cdot x}\,dx\).
Put \(\xi=4\pi n\). The identity
\(P_\ell(x)e^{-\ii\xi\cdot x}
=\ii^\ell P_\ell(\partial_\xi)e^{-\ii\xi\cdot x}\) gives
\(\Lambda_{P_\ell}(n)=
\left.\ii^\ell P_\ell(\partial_\xi)
\widehat{\mathbf1}_{D_R}(\xi)\right|_{\xi=4\pi n}\).
The cube and ball transforms in four dimensions are
\[
\widehat{\mathbf1}_C(\xi)=\prod_{j=1}^4
\frac{2\sin(\xi_j/4)}{\xi_j},
\qquad
\widehat{\mathbf1}_{B_R}(\xi)=4\pi^2R^2
\frac{J_2(R|\xi|)}{|\xi|^2}.
\]
The apparent singularities at \(\xi_j=0\) and \(\xi=0\) are
removable: the corresponding limits are \(1/2\) and
\(\pi^2R^4/2\), respectively.
Consequently,
\begin{equation}\label{app:eq:Fourier-moment}
\Lambda_{P_\ell}(n)=\left.\ii^\ell P_\ell(\partial_\xi)
\left[\prod_{j=1}^4\frac{2\sin(\xi_j/4)}{\xi_j}
-4\pi^2R^2\frac{J_2(R|\xi|)}{|\xi|^2}\right]
\right|_{\xi=4\pi n}.
\end{equation}
For \(\ell=0\), \(|D_R|=1/16-\pi^2R^4/2\). If \(n\ne0\),
the cube contribution vanishes at \(\xi=4\pi n\), giving
\(\Lambda_1(n)=-R^2J_2(4\pi R|n|)/(4|n|^2)\).

\subsection{Explicitness of the periodic data}

\begin{proposition}[Explicit periodic data]
\label{app:prop:explicit-periodic}
Every periodic quantity entering the construction of the exterior
Dirichlet and harmonic operators is determined by the following
operations:
\begin{enumerate}
\item the Ewald representation \eqref{app:eq:Ewald};
\item termwise differentiation justified by
Proposition~\ref{app:prop:Ewald-derivatives};
\item the normalized multipole formula \eqref{app:eq:multipole};
\item the Fourier-moment formula \eqref{app:eq:Fourier-moment};
\item {projection onto any fixed finite set of spherical harmonics.}
\end{enumerate}
Every Ewald truncation has an exponentially decaying remainder.
\end{proposition}
\section{Radial Green kernels on the Kummer caps}
\label{app:EH-kernels}

We construct the separated Green kernels on the sixteen Kummer
caps. Each cap metric is radial, agrees with Eguchi--Hanson near its
exceptional curve, and is flat in a fixed boundary collar.

We work on one cap and suppress its index \(\nu\). We complexify the
scalar operator and use the Hermitian angular inner product.
Complex conjugation interchanges the \((p,q)\) and \((q,p)\) spaces,
while the radial coefficients depend only on \(p+q\) and \(pq\).
Conjugate modes therefore give real solutions for real sources and
boundary values.

\subsection{Radial K\"ahler metrics}

For a radial potential \(\Phi(\rho)\), where
\(\rho=|z_1|^2+|z_2|^2\), set \(A=\Phi'\) and
\(B=\Phi'+\rho\Phi''\). The metric and its inverse are
\begin{equation}\label{app:eq:radial-inverse}
\begin{aligned}
g_{i\bar j}&=A\delta_{ij}+A'\bar z_i z_j,\\
g^{i\bar j}&=\frac1A\delta_{ij}-\frac{A'}{AB}z_i\bar z_j.
\end{aligned}
\end{equation}
Its volume form is \(\omega^2/2=4AB\,dV_{\mathbb R^4}\).
For the background of Section~\ref{subsec:radial-interpolation},
\[
A=A_a^{\mathrm{mod}}=\frac{\sqrt{D_a(\rho)}}{2\rho},
\qquad
B=B_a^{\mathrm{mod}}=\frac{\rho(1-\beta_a(\rho))}{2\sqrt{D_a(\rho)}},
\qquad
4AB=1-\beta_a.
\]
Near the exceptional divisor these reduce to
\[
A=\frac{\sqrt{\rho^2+a^4}}{2\rho},
\qquad
B=\frac{\rho}{2\sqrt{\rho^2+a^4}}.
\]
On the fixed outer collar \(r_4\le r\le R\), one has
\(A=B=1/2\).

\subsection{Bihomogeneous harmonic modes}

Let \(P_{p,q,\gamma}\) be a bihomogeneous Euclidean harmonic polynomial
of bidegree \((p,q)\), with \(p,q\) nonnegative integers.
With \(\ell=p+q\), write
\(P_{p,q,\gamma}(r\theta)=r^\ell Y_{p,q,\gamma}(\theta)\).
Only even \(\ell\) descend to the Kummer quotient. Bihomogeneity gives
\begin{equation}\label{app:eq:Euler}
\sum_jz_j\partial_{z_j}P_{p,q,\gamma}=pP_{p,q,\gamma},
\qquad
\sum_j\bar z_j\partial_{\bar z_j}P_{p,q,\gamma}=qP_{p,q,\gamma},
\end{equation}
and harmonicity gives
\(\sum_j\partial_{z_j}\partial_{\bar z_j}P_{p,q,\gamma}=0\).
We seek solutions of the form \(u=w(\rho)P_{p,q,\gamma}\).

\begin{proposition}[Separated radial equation]
\label{app:prop:radial-equation}
One has
\begin{equation}\label{app:eq:radial-Lambda}
\Lambda_\omega\ddc u
=
\frac{\rho Aw''+\bigl((\ell+2)A+\rho A'\bigr)w'-pqA'w}{AB}
P_{p,q,\gamma}.
\end{equation}
Thus \(-\Lambda_\omega\ddc(wP_{p,q,\gamma})
=F(\rho)P_{p,q,\gamma}\) is equivalent to
\begin{equation}\label{app:eq:radial-SL}
-\left(\rho^{\ell+2}Aw'\right)'
+pq\,\rho^{\ell+1}A'w
=\rho^{\ell+1}AB\,F.
\end{equation}
\end{proposition}

\begin{proof}
Writing \(P=P_{p,q,\gamma}\), differentiation gives
\[
u_{i\bar j}
=w''\bar z_i z_jP
+w'\bigl(\delta_{ij}P+\bar z_iP_{\bar j}+z_jP_i\bigr)
+wP_{i\bar j}.
\]
Contracting with \eqref{app:eq:radial-inverse} and using
\eqref{app:eq:Euler} gives \eqref{app:eq:radial-Lambda}.
Multiplication by \(\rho^{\ell+1}AB\) then gives
\eqref{app:eq:radial-SL}.
\end{proof}

Set \(\mathsf p_{\ell,a}=\rho^{\ell+2}A\) and
\(\mathsf q_{p,q,a}=pq\,\rho^{\ell+1}A'\). The homogeneous equation is
\begin{equation}\label{app:eq:homogeneous-radial}
-(\mathsf p_{\ell,a}w')'+\mathsf q_{p,q,a}w=0.
\end{equation}
\subsection{The regular solution at the exceptional divisor}

As \(\rho\to0\),
\(A(\rho)=\frac{a^2}{2\rho}(1+O(\rho^2/a^4))\).
The indicial equation of \eqref{app:eq:homogeneous-radial} is
\(s(s+\ell)+pq=(s+p)(s+q)=0\), with roots \(-p\) and \(-q\).
Set \(m=\min\{p,q\}\) and \(M=\max\{p,q\}\).

\begin{lemma}[Regular radial solution]
\label{app:lem:regular-Frobenius}
Up to a multiplicative constant, there is a unique solution
\(w_{p,q,a}^{\mathrm{reg}}\) of
\eqref{app:eq:homogeneous-radial} for which
\(w_{p,q,a}^{\mathrm{reg}}P_{p,q,\gamma}\) extends smoothly across
the exceptional divisor. Normalize it by
\(w_{p,q,a}^{\mathrm{reg}}(\rho)
=\rho^{-m}(1+O(\rho^2/a^4))\) as \(\rho\to0\).
In the exact Eguchi--Hanson region it has the Frobenius expansion
\begin{equation}\label{app:eq:Frobenius}
w_{p,q,a}^{\mathrm{reg}}(\rho)
=\rho^{-m}\sum_{j=0}^{\infty}
c_j^{p,q}\left(\frac{\rho^2}{a^4}\right)^j,
\qquad c_0^{p,q}=1,
\end{equation}
convergent for \(0<\rho<\min\{a^2,r_1^2\}\), with coefficients
\begin{equation}\label{app:eq:Frobenius-recursion}
c_j^{p,q}
=-\frac{(2j-m-2)(2j+M-1)}{2j(2j+M-m)}c_{j-1}^{p,q},
\qquad j\ge1.
\end{equation}
\end{lemma}

\begin{proof}
The origin is a regular singular point of
\eqref{app:eq:homogeneous-radial}. Its indicial roots are \(-p\)
and \(-q\). In a resolution chart, the branch corresponding to
\(-m\) extends smoothly across the exceptional divisor. The other
branch is singular; when \(p=q\), the second independent solution
has a logarithmic singularity.

In the exact Eguchi--Hanson region the equation becomes
\[
\rho^2(\rho^2+a^4)w''
+\rho\bigl((\ell+2)\rho^2+(\ell+1)a^4\bigr)w'
+pq\,a^4w=0.
\]
Substituting \eqref{app:eq:Frobenius} gives
\eqref{app:eq:Frobenius-recursion}, whose denominator is nonzero
for \(j\ge1\). The ratio test gives convergence for \(\rho<a^2\),
with possible termination; the background uses this equation only
for \(\rho<r_1^2\).

In a resolution chart \((\zeta,t)\), where
\(\rho=|t|(1+|\zeta|^2)\), each monomial of
\(\rho^{-m}P_{p,q,\gamma}\) is \(t^{(p-q)/2}\), when \(p\ge q\),
or its conjugate when \(q\ge p\), times a smooth function of
\(\zeta\). These exponents are integers because \(\ell\) is even,
and every additional factor \(\rho^{2j}\) is smooth. {For \(p\ne q\), the other branch has a negative power of the
normal fiber coordinate, possibly with a logarithmic term. When
\(p=q\), the second solution contains a logarithm. In either
case it cannot extend smoothly across the exceptional divisor.}
\end{proof}

Continue the regular solution through the fixed transition annulus
using \eqref{app:eq:homogeneous-radial}. This continuation also has
a convergent iterated-integral representation. Choose
\(0<\rho_*<\min\{a^2,r_1^2\}\) and set
\[
V=\binom{w}{\mathsf p_{\ell,a}w'},
\qquad
\mathsf C(\rho)=
\begin{pmatrix}
0&\mathsf p_{\ell,a}^{-1}\\
\mathsf q_{p,q,a}&0
\end{pmatrix}.
\]
Then \(V'=\mathsf C V\), and for \(\rho_*\le\rho\le R^2\),
\begin{equation}\label{app:eq:Volterra-series}
V(\rho)=\left[I+\sum_{j=1}^{\infty}
\int_{\rho_*<s_j<\cdots<s_1<\rho}
\mathsf C(s_1)\cdots\mathsf C(s_j)\,ds_j\cdots ds_1\right]V(\rho_*).
\end{equation}
With \(L_{p,q,a}=\int_{\rho_*}^{R^2}\|\mathsf C(s)\|\,ds\),
the tail after order \(J\) is bounded by
\[
\|V(\rho_*)\|e^{L_{p,q,a}}
\frac{L_{p,q,a}^{J+1}}{(J+1)!}.
\]
This bound applies to each fixed \(a,p,q\); it is not uniform in
the angular modes.

\begin{lemma}[Nonvanishing]
\label{app:lem:regular-nonvanishing}
The regular solution satisfies
\(w_{p,q,a}^{\mathrm{reg}}(\rho)\ne0\) for \(0<\rho\le R^2\).
\end{lemma}

\begin{proof}
If \(w_{p,q,a}^{\mathrm{reg}}(\rho_0)=0\) for some \(\rho_0>0\),
then \(w_{p,q,a}^{\mathrm{reg}}P_{p,q,\gamma}\) would be a nonzero
smooth harmonic function on the resolved region bounded by
\(r^2=\rho_0\), with zero Dirichlet boundary value. The Dirichlet
energy identity forces this function to vanish, a contradiction.
\end{proof}

\subsection{The zero-boundary second solution}

Reduction of order gives
\begin{equation}\label{app:eq:wD}
w_{p,q,a}^{D}(\rho)
=w_{p,q,a}^{\mathrm{reg}}(\rho)
\int_\rho^{R^2}
\frac{ds}{s^{\ell+2}A(s)
\bigl(w_{p,q,a}^{\mathrm{reg}}(s)\bigr)^2}.
\end{equation}
This solution vanishes at \(R^2\) and has the normalized Wronskian
\[
\rho^{\ell+2}A(\rho)
\left[w_{p,q,a}^{D}(w_{p,q,a}^{\mathrm{reg}})'
-(w_{p,q,a}^{D})'w_{p,q,a}^{\mathrm{reg}}\right]=1.
\]
\subsection{The radial Dirichlet Green kernel}

For \(\rho_<=\min\{\rho,\sigma\}\) and
\(\rho_>=\max\{\rho,\sigma\}\), define
\begin{equation}\label{app:eq:radial-Green}
g_{p,q,a}^{D}(\rho,\sigma)
=w_{p,q,a}^{\mathrm{reg}}(\rho_<)w_{p,q,a}^{D}(\rho_>).
\end{equation}
\begin{proposition}[Radial Green formula]
\label{app:prop:radial-Green}
Suppose that \(F(\rho)P_{p,q,\gamma}\) extends smoothly over the
resolved cap, including its exceptional divisor and boundary.
The unique smooth solution \(wP_{p,q,\gamma}\) of
\(-\Lambda_{\omega_a^{\mathrm{mod}}}\ddc(wP_{p,q,\gamma})
=F(\rho)P_{p,q,\gamma}\), with zero boundary value at \(r=R\),
is given by
\begin{equation}\label{app:eq:radial-Green-solution}
w(\rho)=\int_0^{R^2}g_{p,q,a}^{D}(\rho,\sigma)
\sigma^{\ell+1}A(\sigma)B(\sigma)F(\sigma)\,d\sigma.
\end{equation}
\end{proposition}

\begin{proof}
Away from \(\rho=\sigma\), the kernel solves the homogeneous
equation. The Wronskian identity gives
\[
\partial_\rho g(\sigma^+,\sigma)-\partial_\rho g(\sigma^-,\sigma)
=-\frac1{\sigma^{\ell+2}A(\sigma)}.
\]
Thus \(-[\sigma^{\ell+2}A(\sigma)\partial_\rho g]_{
\sigma^-}^{\sigma^+}=1\), as required by
\eqref{app:eq:radial-SL}. The zero boundary condition follows from
\(w^D(R^2)=0\). For the stated source, the corresponding angular
component of the smooth Dirichlet solution on the resolved cap
satisfies this equation and the regular endpoint condition.
Variation of parameters with the normalized Wronskian gives
\eqref{app:eq:radial-Green-solution}, and Dirichlet uniqueness
identifies the expression with that smooth solution.
\end{proof}

\subsection{The full cap Dirichlet inverse}

Let \(L=S^3/\{\pm1\}\), and choose a complex orthonormal basis
\(Y_{p,q,\gamma}\) for its standard quotient measure \(d\sigma_L\),
so that \(\int_LY_I\overline{Y_J}\,d\sigma_L=\delta_{IJ}\).
Expand a source as
\(f(r,\theta)=\sum_{p,q,\gamma}f_{p,q,\gamma}(r)
Y_{p,q,\gamma}(\theta)\).
Since \(P_{p,q,\gamma}=\rho^{\ell/2}Y_{p,q,\gamma}\), its coefficient
in the polynomial normalization is
\(F_{p,q,\gamma}(\rho)
=\rho^{-\ell/2}f_{p,q,\gamma}(\sqrt\rho)\).
The cap inverse is therefore
\begin{equation}\label{app:eq:Dcap}
\begin{aligned}
D_{\mathrm{cap},a}f
&=\sum_{p,q,\gamma}w_{p,q,\gamma}(\rho)P_{p,q,\gamma},\\
w_{p,q,\gamma}(\rho)
&=\int_0^{R^2}g_{p,q,a}^{D}(\rho,\sigma)
\sigma^{\ell+1}A(\sigma)B(\sigma)F_{p,q,\gamma}(\sigma)\,d\sigma.
\end{aligned}
\end{equation}
The volume measure is
\(dV_a=2\rho A(\rho)B(\rho)\,d\rho\,d\sigma_L\), so the full kernel is
\[
G_{\mathrm{cap},a}^{D}(x,y)
=\frac12\sum_{p,q,\gamma}
(\rho_x\rho_y)^{\ell/2}g_{p,q,a}^{D}(\rho_x,\rho_y)
Y_{p,q,\gamma}(\theta_x)\overline{Y_{p,q,\gamma}(\theta_y)}.
\]
This formula is initially interpreted distributionally. For
\(\rho_x\ne\rho_y\), {radial separation gives locally uniform absolute convergence,
including differentiated series away from the endpoints.}
{At equal positive radii}, use Abel regularization by inserting
\(e^{-t\ell(\ell+2)}\) and letting \(t\downarrow0\); the regularized
kernel converges to the Green kernel off the diagonal. Absolute
convergence at equal radii does not follow from \(x\ne y\) alone.
{
At the exceptional divisor, angular coordinates degenerate, and
this Abel rule is not a pointwise prescription at radius zero.
There the inverse is evaluated after integration against a smooth
source, as justified below. Off the diagonal its kernel values may
equivalently be obtained by the smooth extension in resolution
charts; the positive-radius Abel limit is taken before approaching
the exceptional divisor.}
{For smooth sources, the following angular projection argument
justifies synthesis on the entire resolved cap.}
{
To justify this synthesis at the exceptional divisor, fix \(a>0\)
and let \(\mathsf P_{p,q}\) be the character projector onto the
corresponding irreducible \(U(2)\) representation. The group
action extends smoothly to the resolved cap, preserves the metric
and boundary, and commutes with the scalar operator. Let
\(\mathscr C\) be the nonnegative \(SU(2)\) Casimir normalized
to have eigenvalue \(\ell(\ell+2)\) on this representation.
Since its dimension is \(\ell+1\), the Haar projector formula,
character orthogonality, and invariance of the tensor norms give
\[
\|\mathsf P_{p,q}v\|_{C^k(g_a)}
\le \frac{\ell+1}{(1+\ell(\ell+2))^N}
\|(1+\mathscr C)^Nv\|_{C^k(g_a)}
\]
for smooth \(v\) and integers \(k,N\ge0\). There are
\(\ell+1\) bidegrees at degree \(\ell\), so choosing \(N\)
sufficiently large gives absolute convergence of the grouped
angular sum in every fixed \(C^k\)-norm. Apply this to the
unique smooth Dirichlet solution for a smooth source. Each
projection has the regular endpoint condition and is represented
by \eqref{app:eq:radial-Green-solution}; hence angular synthesis
recovers that solution on the entire resolved cap. Constants in
this argument may depend on \(a\), and no uniform high-order
truncation bound is inferred. For H\"older sources, approximation
in \(C^{0,\beta}\), \(0<\beta<\alpha\), followed by uniqueness
and the Schauder estimate gives the same inverse; density in the
full \(C^{0,\alpha}\)-norm is not required.}

\subsection{Harmonic extension and the Dirichlet-to-Neumann map}

The regular harmonic extension of a boundary mode is
\begin{equation}\label{app:eq:Hcap}
H_{\mathrm{cap},a}Y_{p,q,\gamma}
=\frac{r^\ell w_{p,q,a}^{\mathrm{reg}}(r^2)}
{R^\ell w_{p,q,a}^{\mathrm{reg}}(R^2)}Y_{p,q,\gamma}.
\end{equation}
The metric is flat near \(r=R\), where the outward unit normal is
\(\partial_r\). Its Dirichlet-to-Neumann eigenvalue is therefore
\begin{equation}\label{app:eq:cap-DtN}
\lambda_{p,q}^{\mathrm{cap}}(a)
=\frac\ell R+2R\frac{(w_{p,q,a}^{\mathrm{reg}})'(R^2)}
{w_{p,q,a}^{\mathrm{reg}}(R^2)}.
\end{equation}
Equivalently, with
\(\mathcal U_{p,q,a}(r)=r^\ell w_{p,q,a}^{\mathrm{reg}}(r^2)\),
this is \(\partial_r\log\mathcal U_{p,q,a}(r)|_{r=R}\).

\subsection{The fixed flat collar}

On \(r_4\le r\le R\), the homogeneous equation reduces to
\(\rho w''+(\ell+2)w'=0\), so
\[
w(\rho)=c_++c_-\rho^{-\ell-1},
\qquad
\mathcal U(r)=c_+r^\ell+c_-r^{-\ell-2}.
\]
With the dimensionless reflection coefficient
\(\mathfrak r_{p,q}(a)=(c_-/c_+)R^{-2\ell-2}\), the boundary
coefficient becomes
\begin{equation}\label{app:eq:DtN-reflection}
\lambda_{p,q}^{\mathrm{cap}}(a)
=\frac1R\frac{\ell-(\ell+2)\mathfrak r_{p,q}(a)}
{1+\mathfrak r_{p,q}(a)}.
\end{equation}
The interior Dirichlet-to-Neumann coefficient \(\lambda_4\) at
\(r_4\) is nonnegative by the energy identity. Collar matching gives
\[
\frac{c_-}{c_+}r_4^{-2\ell-2}
=\frac{\ell-r_4\lambda_4}{\ell+2+r_4\lambda_4}.
\]
In particular, \(c_+\ne0\) and
\(|\mathfrak r_{p,q}(a)|\le(r_4/R)^{2\ell+2}\).
This bound gives the uniform smoothing used in
Section~\ref{subsec:uniform-high-modes} and controls the denominator
in \eqref{app:eq:DtN-reflection}.

\subsection{The flat-orbifold limit}

We record the fixed-mode geometric limit of the cap data.
The uniform inverse estimate in Section~\ref{subsec:uniform-low-modes}
uses exterior coercivity and does not require this limiting argument.

\begin{proposition}[Fixed-mode flat limit]
\label{app:prop:flat-limit}
For every fixed \(p,q,\gamma\),
\(\lambda_{p,q}^{\mathrm{cap}}(a)\to\ell/R\) as \(a\to0\), or,
equivalently, \(\mathfrak r_{p,q}(a)\to0\).
\end{proposition}

\begin{proof}
Let \(u_a\) be the harmonic extension of \(Y_{p,q,\gamma}\) from
\(r=R\). The maximum principle, applied to the subharmonic function
\(|u_a|^2\), gives a uniform supremum bound from this fixed boundary
value. Away from the exceptional point the metrics converge smoothly
to the flat orbifold metric. Local elliptic estimates therefore give
subsequential smooth convergence \(u_a\to u_0\) on compact subsets
of the punctured ball. On the fixed flat boundary collar, boundary
Schauder estimates also give convergence up to \(r=R\), so the limit
has boundary value \(Y_{p,q,\gamma}\).

The limit is bounded, so its lift to the double cover has a removable
harmonic singularity at the origin. It is consequently the regular
flat harmonic extension
\[
u_0(r,\theta)=(r/R)^\ell Y_{p,q,\gamma}(\theta).
\]
This identifies the limit independently of the subsequence. The same
boundary estimates give convergence of the normal derivatives at
\(r=R\), proving the asserted limit.
\end{proof}

\subsection{Summary of the cap construction}

\begin{proposition}[Explicit cap data]
\label{app:prop:explicit-cap-data}
For sufficiently small \(a>0\), the cap operators used in
Section~\ref{sec:green-construction} are determined by the regular
radial solutions, obtained from the convergent Frobenius expansion
and one-dimensional continuation, and the zero-boundary solutions
\eqref{app:eq:wD}. These give the radial Green kernels
\eqref{app:eq:radial-Green}, the Dirichlet inverse
\eqref{app:eq:Dcap}, the harmonic extension \eqref{app:eq:Hcap},
and the Dirichlet-to-Neumann coefficients \eqref{app:eq:cap-DtN}.
After angular projection, their construction uses one-dimensional
differentiation and integration, convergent radial series, and angular
synthesis.
\end{proposition}
\section{Angular bases and Fourier moments}
\label{app:angular-evaluation}
The following optional algebraic reductions assist evaluation of
the representation. They are not needed for the Green estimate
or the convergence theorem. We give the symmetry reduction,
harmonic projection, measure normalization, and polynomial--Fourier
moments together.

\subsection{Symmetries preserving the complex structure}
\label{subsec:Kummer-symmetries}
The holomorphic unitary lattice group used here is
\[
\Gamma=\{(z_1,z_2)\mapsto
(\zeta_1z_{\sigma(1)},\zeta_2z_{\sigma(2)}):
\zeta_j\in\{1,-1,\ii,-\ii\},\ \sigma\in S_2\}.
\]
Its action lifts to the resolution and preserves the radial
background, its volume form, and the equation. Translation by each
two-torsion point also lifts and permutes the sixteen identical caps.
Uniqueness of the mean-zero inverse, or induction on the tree
recursion, shows that \(u_a\) has these symmetries.
The larger real signed-permutation group \(B_4\) preserves the flat
torus Green kernel, but its arbitrary real-coordinate permutations
need not preserve the complex structure or the Eguchi--Hanson metric.
It cannot be used to discard modes in the full cap or nonlinear problem.
The Green operator for arbitrary sources uses the entire even harmonic
basis. For the \(\Gamma\)-invariant sources arising in the correction
recursion, one may instead use the invariant basis constructed here.

\subsection{Exact harmonic projection}
\label{subsec:exact-harmonic-projection}

Let \(P\) be a homogeneous polynomial of degree \(\ell\) on
\(\mathbb R^4\).  Its harmonic projection is
\[
\Pi_\ell P
=
\sum_{j=0}^{\lfloor\ell/2\rfloor}
(-1)^j
\frac{
(\ell-j)!
}{
4^jj!\,\ell!
}
|x|^{2j}\Delta_0^jP.
\]
Thus harmonic projection is a finite algebraic operation.

For a real homogeneous polynomial \(P\), form the finite group average
\(\mathcal R_\Gamma P(x)=|\Gamma|^{-1}\sum_{g\in\Gamma}P(gx)\).
Because \(\Gamma\) is orthogonal, harmonic projection commutes with
this average. A spanning family for its invariant harmonics of
degree \(\ell\) is therefore
\(\{\Pi_\ell\mathcal R_\Gamma x^\nu:|\nu|=\ell\}\).
Remove linear dependencies by exact row reduction before forming
the Gram matrix. All coefficients are rational. The absence of
invariant harmonic quadratics follows because independent phase
rotations eliminate the non-Hermitian quadratic terms, while coordinate
interchange leaves only a multiple of \(|z_1|^2+|z_2|^2\), which is
{not harmonic.}

\subsection{Exact spherical Gram matrices}
\label{subsec:exact-Gram}

For stable reconstruction it is convenient to orthogonalize the
canonical harmonic basis exactly.

Normalize
\begin{equation}\label{eq:normalized-sphere-inner}
\langle f,g\rangle_S
=
\frac1{2\pi^2}
\int_{S^3}
fg\,d\sigma.
\end{equation}
If \(a_1,\ldots,a_4\) are nonnegative integers and
\(A=a_1+\cdots+a_4\), then
\[
\frac1{2\pi^2}
\int_{S^3}
x_1^{2a_1}\cdots x_4^{2a_4}\,d\sigma
=
\frac1{(A+1)!}
\prod_{j=1}^4
\frac{
(2a_j)!
}{
4^{a_j}a_j!
}.
\]
Monomials with an odd exponent have zero integral. Hence every Gram
matrix \(G_\ell=\left(\langle H_I,H_J\rangle_S\right)_{I,J}\)
has rational entries.

Choose an exact factorization \(G_\ell=L_\ell D_\ell L_\ell^T\).
Then \(E_\ell=D_\ell^{-1/2}L_\ell^{-1}H_\ell\) is orthonormal
for \eqref{eq:normalized-sphere-inner}.
For the standard quotient measure on \(S^3/\{\pm1\}\), whose
total volume is \(\pi^2\), the corresponding basis is
\(Y_I=E_I/\pi\). For the physical boundary measure on
\(\Sigma_\nu\), the basis becomes \(R^{-3/2}Y_I\).
These rescalings reconcile the normalized Gram calculation with
the angular kernels and the matrix \eqref{eq:M-a-L}.
To apply the separated cap formulas, decompose this real basis into the complex \((p,q)\) angular spaces.

{This helps control errors when reconstructing coefficients from
a poorly conditioned monomial basis.}

\subsection{Fourier moments}
\label{subsec:computational-Fourier-moments}

The fully Kummer-symmetric Fourier expansion uses frequencies
\(4\pi n\), where \(n\in\mathbb Z^4\). Let \(C=[-1/4,1/4]^4\)
and \(D_R=C\setminus B_R\). For a homogeneous harmonic polynomial
\(H\) of degree \(\ell\), put
\(\Lambda_H(n)=\int_{D_R}H(x)e^{-4\pi\ii n\cdot x}\,dx\).
Formula~\eqref{app:eq:Fourier-moment}, with \(P_\ell=H\), evaluates
these polynomial--Fourier volume moments by differentiation of the
explicit elementary--Bessel expression at \(\xi=4\pi n\).
Boundary matrix elements are computed from the separate periodic
trace and normal-derivative formulas.

{On the flat exterior, a Fourier calculation requires a specified
periodic extension of the functions from the punctured torus.
If the chosen extensions have sufficient regularity to justify
Fourier multiplication, the coefficients of a nonlinear source
can be assembled by convolution and then used for all required
harmonic moments. This avoids repeating the convolution for
each harmonic channel. The extension and its truncation errors
must be included in any numerical certificate.}

\end{document}